\documentclass[11pt,reqno]{amsart}
\usepackage[utf8x]{inputenc}
\usepackage[T1]{fontenc}
\usepackage{amsmath,amsthm}
\usepackage{amsfonts,amssymb}
\usepackage{a4wide}

\newtheorem{theorem}{Theorem}
\newtheorem{lemma}{Lemma}[section]
\newtheorem{proposition}{Proposition}[section]
\newtheorem{claim}{Claim}

\theoremstyle{remark}
\numberwithin{equation}{section}

\newcommand{\R}{\mathbb{R}}
\newcommand{\N}{\mathbb{N}}
\numberwithin{equation}{section}
\def\bell{{\boldsymbol{\ell}}}

\def\psl#1#2{\left(#1,#2 \right)_{L^2}}
\def\pse#1#2{\left(#1,#2 \right)_E}
\def\psh#1#2{\left(#1,#2 \right)_{\dot H^1}}
\def\pshb#1#2{\left(#1,#2 \right)_{\dot H^1_\ell}}
\def\pshbb#1#2{\left(#1,#2 \right)_{\dot H^1_\bell}}
\def\pshbbk#1#2{\left(#1,#2 \right)_{\dot H^1_{\bell_k}}}
\def\pshbbun#1#2{\left(#1,#2 \right)_{\dot H^1_{\bell_1}}}
\def\nol#1{\left\|#1 \right\|_{L^2}}
\def\noe#1{\left\|#1 \right\|_E}
\def\noh#1{\left\|#1 \right\|_{\dot H^1}}
\def\nohb#1{\left\|#1 \right\|_{\dot H^1_\ell}}
\def\nohbb#1{\left\|#1 \right\|_{\dot H^1_\bell}}
\def\fl#1{\vec{#1}}

\def\e{\varepsilon}
\def\Nint{\mathcal N_{\Omega}}
\def\Nsol{\mathcal N_{\Omega^C}}
\def\SWk{\sum_k W_k}
\def\sumk{\mathcal W_K}
\def\osumk{\widetilde {\mathcal W}_K}
\def\lg{\langle}
\def\rg{\rangle}

\begin{document}
\title[Multi-solitons for critical wave equation]{Construction of multi-solitons for the energy-critical   wave equation in  dimension 5}
\author[Y. Martel]{Yvan Martel}
\address{Ecole polytechnique, CMLS CNRS UMR7640, 91128 Palaiseau, France}
\email{yvan.martel@polytechnique.edu}
\author[F. Merle]{Frank Merle}
\address{Universit\'e de Cergy Pontoise and Institut des Hautes \'Etudes Scientifiques, AGM CNRS UMR8088, 95302 Cergy-Pontoise, France}
\email{merle@math.u-cergy.fr}
\begin{abstract}
We construct  $2$-solitons of the focusing energy-critical nonlinear wave equation in space dimension $5$,
i.e. solutions $u$ of the equation  such that
$$
u(t) - \left[ W_1(t) + W_2(t)\right] \to 0 \quad \hbox{as $t\to +\infty$}
$$
 in the energy space,
where $W_1$ and $W_2$ are   Lorentz transforms of the explicit standing soliton   $W(x) = ( 1+  {|x|^2}/{15} )^{-3/2}$, with any speeds
$\bell_1\neq \bell_2$ ($|\bell_k|<1$).
The existence result also holds for  the case of $K$-solitons, for any $K\geq 3$, assuming that the speeds $\bell_k$ are collinear.

The main difficulty of the construction is the strong interaction between the solitons due to the slow algebraic decay of $W(x)$ as $|x|\to +\infty$. This is in contrast with previous constructions of multi-solitons for other nonlinear dispersive equations (like generalized KdV and nonlinear Schr\"odinger equations in energy subcritical cases), where the interactions are exponentially small in time due to the exponential decay of the solitons.
\end{abstract}
\maketitle
\section{Introduction}
\subsection{Statement of the main result}
We consider the focusing energy-critical nonlinear wave equation in  dimension $5$
\begin{equation}\label{wave}
\left\{ \begin{aligned}
&\partial_t^2 u - \Delta u - |u|^{\frac 4{3}} u = 0, \quad (t,x)\in [0,\infty)\times \R^5,\\
& u_{|t=0} = u_0\in \dot H^1,\quad 
\partial_t u_{|t=0} = u_1\in L^2.
\end{aligned}\right.
\end{equation}
Recall that the Cauchy problem for equation \eqref{wave} is locally well-posed in the energy space $\dot H^1\times L^2$, using suitable Strichartz estimates. 
See   e.g. \cite{Pecher84,GiSoVe92,LiSo95,ShSt94,ShSt98,Sogge95,Kapitanski94, KM}.
Note that equation \eqref{wave} is invariant by the $\dot H^1$ scaling:
if $u(t,x)$ is solution of \eqref{wave}, then
$$
u_\lambda(t,x)=\frac{1}{\lambda^{3/2}}u\left(\frac{t}{\lambda},\frac{x}{\lambda}\right)
$$
is also solution of \eqref{wave} and $\|u_\lambda\|_{\dot H^ 1}=\|u\|_{\dot H^ 1}$.
For  $\dot H^1\times L^2$ solution, the energy 
$E(u(t),\partial_t u(t))$ and momentum $M(u(t),\partial_t u(t))$ are conserved, where
$$
E(u,v) = \frac 12 \int v^2 + \frac 12 \int |\nabla u|^2 - \frac {3}{10} \int |u|^{\frac {10}{3}},
\quad
M(u,v)=\int v\nabla u.
$$

Recall that  the function $W$ defined by
\begin{equation}\label{defW}
W(x) = \left( 1+ \frac {|x|^2}{15}\right)^{-\frac{3}2},\quad 
\Delta W + W^{\frac 73}=0 , \quad x\in \R^5,
\end{equation}
is a stationary solution, called \textit{soliton}, of \eqref{wave}.
Using the Lorentz transformation on $W$, we obtain traveling solitons:
for $\boldsymbol{\ell}\in \R^5$, with $|\boldsymbol{\ell}|< 1$, let
\begin{equation}\label{defWbb}
W_{\bell}(x)=W\left(\left(\frac{1}{\sqrt{1-|\boldsymbol{\ell}|^2}}-1\right) \frac{\boldsymbol{\ell}(\boldsymbol{\ell}\cdot x)}{|\boldsymbol\ell|^2} +x\right);
\end{equation}
then $u(t,x)=\pm W_{\boldsymbol{\ell}}(x-{\boldsymbol{\ell}} t)$ is   solution of \eqref{wave}.

\medskip

Recall that an important conjecture in the field says that any global solution of \eqref{wave} decomposes as $t\to +\infty$ as a finite sum of (rescaled and translated) solitons 
plus a radiation (solution of the linear wave equation).
Such a classification was achieved in the radial case in \cite{DKM1} (in space dimension 3)  but is still widely open in the nonradial case (see  \cite{DKM2} and references therein).

In this paper, we address the question of the construction of non trivial  asymptotic behaviors  in the nonradial case. 
In this context, multi-solitons are canonical objects behaving as $t\to \infty$ exactly as the sum of several solitons in the energy space.
The main result of this paper is the existence of   $2$-solitons for \eqref{wave} and of $K$-solitons for $K\geq 3$ for collinear speeds.

\begin{theorem}[Existence of multi-solitons]\label{th:1}
Let $K\geq 2$. For  $k\in\{1,\ldots,K\}$, let $\lambda^\infty_k>0$, ${\mathbf y}^\infty_k\in\R^5$, $\iota_k=\pm 1$ and  $\bell_k \in \R^5$ with $|\bell_k|<1$, $\bell_k\neq \bell_{k'}$ for $k'\neq k$. \\Assume that one of the following assumptions holds
\begin{itemize}
\item[{\rm (A)}]{\rm Two-solitons ($K=2$).}
\item[{\rm (B)}]{\rm Collinear speeds.}
For all $k\in\{1,\ldots,K\}$,  $\bell_k=\ell_k \mathbf{e}_1$ where $\ell_k\in(-1,1)$.
\end{itemize}
Then, there exist  $T_0>0$ and a solution $u$ of \eqref{wave}   on $[T_0,+\infty)$ in the energy space such that
\begin{align}\label{eq:th1}
 &\lim_{t\to +\infty} \left\|u(t) 
-\sum_{k=1}^K \frac{\iota_k}{(\lambda_k^\infty)^{3/2}} W_{\boldsymbol{\ell}_k}\left(\frac {.-{\boldsymbol{\ell}_k} t-\mathbf{y}_k^\infty} {\lambda_k^\infty}  \right)
\right\|_{\dot H^1}  =0,
\\
\label{eq:th1bis}
&  \lim_{t\to +\infty}  \left\|\partial_t u(t) 
+\sum_{k=1}^K \frac{\iota_k}{(\lambda_k^\infty)^{5/2}}\, ({\bell_k}\cdot  \nabla W_{\boldsymbol{\ell}_k})\left(\frac {.-{\boldsymbol{\ell}_k} t-\mathbf{y}_k^\infty} {\lambda_k^\infty}  \right)
\right\|_{L^2} =0.
\end{align}
\end{theorem}

The question of existence and properties of multi-solitons for nonlinear models has a long history starting with the celebrated works of Fermi, Pasta and Ulam \cite{FPU} and Kruskal and Zabusky \cite{KZ}, and closely related to the study of  integrable equations by the inverse scattering transform. We refer in particular to the review work of Miura \cite{Mi} on multi-solitons for the  Korteweg-de Vries  equation and to Zakharov and Shabat \cite{ZS} for multi-solitons of the 1D cubic Schr\"odinger equation. Recall that in integrable cases, these solutions are very special: they are explicit and  behave exactly as the sum of several solitons both at $t\to +\infty$ and $t\to -\infty$. In particular, they describe the collision and interaction of several solitons globally in time, i.e. for all $t\in (-\infty,+\infty)$.

Apart from works on integrable models, there have been several proofs of existence of multi-solitons for nonlinear dispersive equations, starting with \cite{Me} for the $L^2$ critical nonlinear Schr\"odinger equation and \cite{Ma} for the subcritical and critical generalized Korteweg-de Vries equations. Note that \cite{Ma} also contains a   uniqueness result in the energy space, whose proof is specific to KdV type  equations. Concerning existence, the general strategy of these works is to build backwards in time a sequence of approximate solutions satisfying uniform estimates and then to use a compactness argument. In \cite{Ma} and also in \cite{MMnls}, concerning  the subcritical nonlinear Schr\"odinger equation,      uniform estimates are deduced from long time stability arguments, adapted from  the previous works \cite{We} (for single solitons) and \cite{MMT} (for several decoupled solitons). 
Later, the strategy of these works was extended to the case of exponentially unstable solitons, see \cite{CMM} for the construction of  multi-solitons and \cite{Co} for the classification of all multi-solitons of the supercritical generalized KdV equation. In these   papers, the exponential instability is controled through a simple topological argument.

For the Klein-Gordon equation, the strategy was adapted by Cote and Munoz \cite{CMkg} (for real and unstable solitons) and Bellazzini, Ghimenti and Le Coz \cite{BGLC} (for complex, stable solitons). For the water-waves system, see the recent work of Ming, Rousset and Tzvetkov \cite{MRT}.

Note that all the works mentioned before are for exponentially decaying solitons, and thus  exponentially small interactions  as $t\to +\infty$.
The main difficulty of constructing multi-solitons for \eqref{wave} is due to the algebraic decay of $W$, which implies that the solitons have strong interactions, of order $t^{-3}$.
For the Benjamin-Ono equation, multi-solitons exist with solitons behaving algebraically at $\infty$, but they are obtained explicitly using the integrability of the equation (see e.g. \cite{Mat} and \cite{NL}). Stability and asymptotic stability of such multi-solitons is proved in \cite{KMa}, but relying on specific monotonicity formulas  for KdV type equations.
In \cite{KMR}, devoted to the construction of multi-solitons for the Hartree equation, solitons are also decaying algebraically. 
However, in that case, the potential related to the soliton is exponentially decaying, which allows a decoupling facilitating the construction of an approximate solution at order $t^{-M}$ for arbitrarily large $M$. For  $M>M_0$ large enough, an actual solution can then be constructed close to this approximate solution. Such  decoupling is not present in the case of the energy critical wave equation \eqref{wave} and it seems delicate to construct sharp approximate multi-solitons (i.e. at order $t^{-M}$ for large $M$).

\subsection{Comments on Theorem \ref{th:1}.}
(1) Each soliton being exponentially unstable, it can be derived as a consequence of the proof 
that the multi-solitons constructed in Theorem \ref{th:1} are unstable.
Uniqueness of multi-soliton  in the energy space, up to the unstable directions, is an open problem as for the nonlinear Schr\"odinger equation. 
The uniqueness statements in \cite{Ma} and \cite{Co} are specific to KdV-type equations.

The global behavior of $u(t)$ i.e. for $t<T_0$ is an open problem. 
We conjecture that it does not have the multi-soliton behavior as $t\to -\infty$.
We refer to  \cite{MMimrn} for the proof of nonexistence of pure multi-solitons in the case of the (non integrable) quartic generalized Korteweg de Vries equation for a certain range  of speeds.

\smallskip

(2) Dimension $N\geq 6$. We expect that Theorem \ref{th:1}  still holds true  for the energy-critical wave equation for space dimensions $N\geq 6$. Indeed, at the formal level, all the important computations of this paper can be reproduced for $N\geq 6$.
However, the lack of regularity of the nonlinearity create several additional technical difficulties, which we choose not to treat in this paper.
Recall that such difficulties  were overcome for the  Cauchy problem in the energy space in \cite{BCLPZ}.

\smallskip

(3) Dimension 3 and 4. We conjecture that in this case, there exists no multi-soliton in the sense \eqref{eq:th1}--\eqref{eq:th1bis}, for any value of $K\geq 2$.
Heuristically, from the asymptotics as $|x|\to\infty$,  $W(x)\sim  |x|^{2-N}$ in dimension $N$, the interaction between two solitons of different speeds is $t^{2-N}$, i.e. $t^{-1}$ 
in dimension 3, and $t^{-2}$ in dimension 4.
Following our method, these interactions are too strong and create diverging terms in the construction.
However, to prove nonexistence of multi-soliton rigorously, one would need  \emph{a priori} information on
 any multi-soliton, which is an open problem for any dimension $N\geq 3$.

\subsection{Strategy of the proof}
First, we note that  Theorem \ref{th:1} in case (A) follows from case (B) with $K=2$ and the Lorentz transformation. 
See Section 5 for a detailed proof, inspired by arguments in \cite{KM,DKM2}. 

\medskip

The proof of Theorem \ref{th:1} in case (B) follows the  strategy by uniform estimates and compactness introduced in \cite{Ma} and \cite{MMnls}, but due to the algebraic decay of the solitons, proving uniform estimates is more delicate.
For $k\in\{1,\ldots,K\}$, let $\lambda^\infty_k>0$, ${\mathbf y}^\infty_k\in\R^5$ and  $\bell_k \in \R^5$ with $|\bell_k|<1$, $\bell_k\neq \bell_{k'}$ for $k'\neq k$.

Let $S_n\to +\infty$ as ${n\to \infty}$ and, for each $n$, let $u_n$ be the (backwards) solution of \eqref{wave} with data at time $S_n$ 
\begin{equation}\label{da1}
  u_n(S_n,x) 
\sim\sum_{k=1}^K \frac{\iota_k}{(\lambda_k^\infty)^{3/2}} W_{\boldsymbol{\ell}_k}\left(\frac {x-{\boldsymbol{\ell}_k} S_n-\mathbf{y}_k^\infty} {\lambda_k^\infty}  \right) ,
\end{equation}
\begin{equation}\label{da2}
 \partial_t u(S_n,x) 
\sim - \sum_{k=1}^K \frac {\iota_k}{(\lambda_k^\infty)^{5/2}}  ({\bell_k}\cdot  \nabla W_{\boldsymbol{\ell}_k})\left(\frac {x-{\boldsymbol{\ell}_k} S_n-\mathbf{y}_k^\infty} {\lambda_k^\infty}  \right).
\end{equation}
(See \eqref{defun} for a precise definition of $(u_n(S_n),\partial_t u_n(S_n))$.
The goal is to prove the following uniform estimates on the time interval $[T_0,S_n]$,
\begin{equation}\label{es1}
  \left\|u_n(t) 
-\sum_{k=1}^K \frac{\iota_k}{(\lambda_k^\infty)^{3/2}} W_{\boldsymbol{\ell}_k}\left(\frac {.-{\boldsymbol{\ell}_k} t-\mathbf{y}_k^\infty} {\lambda_k^\infty}  \right)
\right\|_{\dot H^1}\lesssim \frac 1 t,
\end{equation}
\begin{equation}\label{es2}
  \left\|\partial_t u_n(t) 
+\sum_{k=1}^K \frac {\iota_k}{(\lambda_k^\infty)^{5/2}}\, {\bell_k}\cdot  \nabla W_{\boldsymbol{\ell}_k}\left(\frac {.-{\boldsymbol{\ell}_k} t-\mathbf{y}_k^\infty} {\lambda_k^\infty}  \right)
\right\|_{L^2}\lesssim \frac 1 t.
\end{equation}
for $T_0$ large independent of $n$. Indeed, the existence of a multi-soliton then follows easily  from  standard compactness arguments
(note that we also obtain bounds on weighted higher order Sobolev norms for $(u_n,\partial_t u_n)$ which facilitate the convergence).
Thus, we now focus on the proof of \eqref{es1}--\eqref{es2}.
Note first that such  long time stability estimates cannot be true for any initial data of the form \eqref{da1}--\eqref{da2}; indeed, to take into account the exponential instability of each soliton $W_{\bell_k}$, we need to adjust the initial condition $(u_n(S_n),\partial_t u_n(S_n))$. This adjustment relies on a simple topological argument on $K$ scalar parameters, first introduced in a similar context in \cite{CMM}.

We introduce $$
\e   = u_n- \sum_k W_k,\quad
\eta  = \partial_t u_n + \sum_k ({\bell_k}\cdot  \nabla W_k),
$$
where
$$
 W_k(t,x)= \frac{{\iota_k}}{\lambda_k^{3/2}(t)} W_{\boldsymbol{\ell}_k}\left(\frac {x-{\boldsymbol{\ell}_k} t-\mathbf{y}_k(t)} {\lambda_k(t)}  \right).
$$
By a standard procedure, in the definition of $W_k$, the modulation  parameters $\lambda_k(t)$ and $\mathbf{y}_k(t)$ are chosen close to $\lambda_k^{\infty}$ and $\mathbf{y}_k^ {\infty}$ in order to obtain suitable orthogonality conditions on $(\e ,\eta )$.
The equation  of $(\e,\eta)$ is thus coupled by equations on $\lambda_k$ and $\mathbf{y}_k$. See Lemma  \ref{le:4}.

The general strategy of the proof of the uniform estimates \eqref{es1}--\eqref{es2} is to use global   functionals that are locally of the form 
\begin{align*}
 \int_{x\sim \bell_k t+\mathbf{y}_k(t)}\textstyle |\nabla\e|^2+|\eta|^2 + 2  \left(\bell_k \cdot\nabla\e\right)\eta -\frac 73 |W_k|^{\frac 43} \e^2,
\end{align*}
around each soliton $W_k$, i.e. in   regions $x\sim \bell_k t+\mathbf{y}_k(t)$.
Note that the coercivity of such functional  under usual orthogonality conditions on $(\e,\eta)$ is  standard.
The difficulty is to ``glue'' these $K$ functionals to obtain a unique global functional on $(\e,\eta)$ which is locally adapted to each soliton $W_k$.

\smallskip

 In case (B) of Theorem \ref{th:1}, we assume $\bell_k=\ell_k \mathbf{e}_1$ and $-1<\ell_1<\ldots<\ell_K<1$.
To prove \eqref{es1}-\eqref{es2},  we introduce the following energy functional  
\begin{align*}
  & \mathcal H_K  =\int {\mathcal E}_K  +   2 \int ( \chi_K(t,x)  \partial_{x_1} \e )  \eta ,
\end{align*}
where ${\mathcal E}_K$ is the following ``linearized energy density''
\begin{equation}\label{i22}
\textstyle
 {\mathcal E}_K =|\nabla\e|^2+|\eta|^2-\frac35\left(\left|\SWk +\e\right|^{\frac{10}{3}}-\left|\SWk\right|^{\frac{10}{3}}-\frac{10}{3}\left|\SWk\right|^{\frac{4}{3}}\left(\SWk\right)\e\right),
\end{equation}
and the bounded function $\chi_K(t,x)$ is equal to $\ell_k$ in a   neighborhood of the soliton $W_k$ and close to  $\frac {x_1}t$ in ``transition regions'' between two solitons     (see \eqref{defchiK} for a precise definition). 
Note that the functional $\mathcal H_K$   is     inspired by the ones used in \cite{Ma} and \cite{MMnls} for the construction of multi-solitons for (gKdV) and (NLS) equations in energy subcritical cases.

\medskip

The functional $\mathcal H_K$ has  the following two important properties  (see Proposition \ref{mainprop} for more precise statements):

\smallskip

\noindent (1)   $\mathcal H_K$  is coercive, in the sense that (up to unstable directions, to be controled separately),
it controls the size of $(\e,\eta)$ in the energy space
$$
 \mathcal H_K\sim \|\e\|_{\dot H^1}^2+\|\eta\|_{L^2}^2.
$$

\noindent (2)  The   variation of $\mathcal H_K$ is  controled   on $[T_0,S_n]$  in the following (weak) sense
\begin{equation}\label{pourH}
 - \frac d{dt} \left(t^{2} \mathcal H_K \right)  \lesssim t^{-3}.
\end{equation}
Note that the term $t^{-3}$ in the right-hand side is related to interactions between solitons.

\smallskip

Therefore, integrating \eqref{pourH} on $[t,S_n]$, from \eqref{da1}--\eqref{da2}, we find the uniform bound, for any $t\in [T_0,S_n]$, 
$$
\|\e\|_{\dot H^1}+\|\eta\|_{L^2} \lesssim t^{-2}.
$$
By time  integration of the   equations of the parameters, the above estimate implies
$$
|\mathbf{y}_k(t)-\mathbf{y}_k^\infty|\lesssim t^{-1},\quad |\lambda_k(t)-\lambda_k^{\infty}|\lesssim t^{-1},
$$
and \eqref{es1}--\eqref{es2} follow.
 
 \subsection*{Acknowledgements}  \quad 

This work was partly supported by the project ERC 291214 BLOWDISOL. 

\section{Preliminaries}
\subsection{Notation}
We denote
$$\psl g {\tilde g}   =\int g \tilde g,\quad \nol g^2 = \int |g|^2,\quad \psh g {\tilde g} =\int \nabla g \cdot \nabla {\tilde g} ,
\quad \noh g^2=\int |\nabla g|^2.$$
For 
$$\fl g = \left(\begin{array}{c}g \\h\end{array}\right),
\ \fl {\tilde g} = \left(\begin{array}{c}\tilde g \\\tilde h\end{array}\right),$$
set
\begin{align*}
 \psl {\fl g} {\fl {\tilde g}}   = \psl g {\tilde g} + \psl h{\tilde h},\quad
 \pse {\fl g} {\fl {\tilde g}}   = \psh g {\tilde g} + \psl h{\tilde h},\quad 
\|\fl g\|_E^2 = \noh g^2 + \nol h^2.
\end{align*}
When $x_1$ is seen as a specific coordinate, denote
$$
\overline x = (x_2,\ldots, x_5),
\quad \overline \nabla g = (\partial_{x_2} g, \ldots, \partial_{x_5} g),
\quad \overline \Delta g = \sum_{j=2}^5 \partial_{x_j}^2 g.
$$
For $-1<\ell<1$,
$$ \pshb g{\tilde g} =(1-\ell^2)\int \partial_{x_1} g \partial_{x_1}\tilde g+  \int \overline \nabla g \cdot \overline \nabla \tilde g, 
\quad \nohb g^2=  \pshb gg$$
More generally, for $\bell \in \R^5$ such that $|\boldsymbol{\ell}|< 1$, 
 $$ \pshbb g{\tilde g} =\int \left[ \nabla g\cdot \nabla {\tilde g} -   (\bell\cdot\nabla g)(\bell\cdot\nabla \tilde g)\right], 
 \quad \nohbb g^2= \noh g^2 - \nol{\bell\cdot \nabla g}^2.$$
Observe that if we define
$$
g_\bell(x) = g\left(\left(\frac{1}{\sqrt{1-|\boldsymbol{\ell}|^2}}-1\right) \frac{\boldsymbol{\ell}(\boldsymbol{\ell}\cdot x)}{|\boldsymbol\ell|^2} +x\right),
$$
and similarly $\tilde g$, $\tilde g_\bell$, 
then
\begin{equation}\label{Tbbeta}
\pshbb {g_\bell}{\tilde g_\bell}= (1-|\bell|^ 2)^{\frac 12}\psh g{\tilde g}.
\end{equation}

Let $\Lambda$ and $\widetilde \Lambda$ be the $\dot H^1$ and $L^2$ scaling operators  defined as follows
\begin{equation}
\label{aL}
\Lambda g = \frac 32 g+ x \cdot \nabla g,
\quad \widetilde \Lambda g = \frac 5 2 g + x \cdot \nabla g,\quad
\widetilde\Lambda \nabla=\nabla\Lambda,\quad 
\fl \Lambda = \left(\begin{array}{c}\widetilde \Lambda  \\ \Lambda\end{array}\right).
\end{equation}
Let
$$
J=\left(\begin{array}{cc}0 & 1 \\-1 & 0\end{array}\right).
$$
 
Recall the  Hardy  and   Sobolev inequaliies, for any $v\in \dot H^1$,
\begin{equation}\label{z2}
\int \frac {|v|^2}{|x|^2} \lesssim \int |\nabla v|^2,
\end{equation}
\begin{equation}\label{z1}
\|v\|_{L^{10/3}} \lesssim  \|\nabla v\|_{L^2}.
\end{equation}
Set $\lg x \rg = (1+|x|^2)^{\frac 12}$ and 
$$
\|v\|_{Y^0}^2 = \int \left(|v(x)|^2 + |\nabla v(x)|^2\right) \lg x\rg  dx,
\quad
\|v\|_{Y^1}^2 = \int \left(|\nabla v(x)|^2 + |\nabla^2 v(x)|^2\right) \lg x \rg  dx.
$$
If $g\in C([t_1,t_2],Y^0)$ then the unique solution $v\in C([t_1,t_2],\dot H^1)$ of
$\partial_t^2 v - \Delta v = g$ with $v(t_1)=0$ and $\partial_t v(t_1)=0$, satisfies $(v,v_t) \in C([t_1,t_2],Y^1\times Y^0)$ and 
\begin{equation}\label{EEn}
\|(v,v_t)(t)\|_{Y^1\times Y^0} \leq \int_{t_1}^t \|g(s)\|_{Y^0} ds.
\end{equation}
Moreover, the following estimate holds, for all $v\in Y^1$,
\begin{equation}\label{esob}
\||v|^{\frac 43} v\|_{Y^0} \lesssim \|v\|_{\dot H^1}^{\frac 13}\|v\|_{Y^1}^{2}.
\end{equation}
Thus, it follows from a standard argument (fixed point) that  \eqref{wave} is locally well-posed in the space $Y^1\times Y^0$ with a time of existence depending only on the size of the $Y^1\times Y^0$   norm of the initial data.

For initial data in the energy space $\dot H^1\times L^2$, the Cauchy problem is also locally well-posed in a certain sense, using suitable Strichartz estimates ; we refer to section 2 of \cite{KM} and references therein.

Denote
$$
f(u)= |u|^{\frac 43} u,\quad F(u) = \frac {3}{10} |u|^{\frac {10}3}.
$$

\subsection{Energy linearization  around $W$}
Let
\begin{align*}
& L   = -\Delta   - f'(W) ,\quad 
\psl{L g}g =   \int |\nabla g|^2 -  f'(W) g^2,
\\
& H = \left(\begin{array}{cc} L & 0 \\0 & {\rm Id}\end{array}\right),\quad 
\psl{H \fl g} {\fl g} = \psl{L g}g+  \nol h^2.  
\end{align*}
Let  $\fl g$ be small in the energy space.
Then, expanding, integrating by parts, using the equation of $W$ and \eqref{z1}, one has
\begin{align}
E(W+g,h) & = E(W,0) -\int (\Delta W +f(W)) g
+\frac 12 \left( \int |\nabla g|^2 -f'(W) g^2\right)
 +\frac 12 \int h^2
\nonumber \\
&-  \int \left( F(W+g)- F(W)
-f(W) g -\frac 12 f'(W) g^2\right)\nonumber\\
& =E(W,0)+\frac 12\psl{L g}g+ \frac 12 \nol h^2 + O(\noh g^3).\label{enerlin}
\end{align}
In this paper addressing the case of several solitons, it is crucial to be able to spacially split the solitons. For some $0<\alpha \ll1$ to be fixed, set
\begin{equation}\label{phia}
\varphi (x) = (1+|x|^2)^{- \alpha}
\end{equation}

We gather here some properties of the operator $L$.

\begin{lemma}[Spectral properties of $L$]\label{le:Q}
\emph{(i) Spectrum.} The operator $L$ on $L^2$ with domain $H^2$ is a self-adjoint operator with essential spectrum $[0,+\infty)$, no positive eigenvalue and only one negative eigenvalue $-\lambda_0$, with a smooth radial positive eigenfunction $Y \in \mathcal S(\R^5)$.
Moreover,
\begin{equation}
\label{ker}
L (\Lambda W) = L (\partial_{x_j} W) =0, \quad \hbox{for any $j=1,\ldots,5$.}
\end{equation}

There exists $\mu>0$ such that, for all $g \in \dot H^1$, the following holds.\\
\emph{(ii) Coercivity with $W$ orthogonality (Appendix D of \cite{OR}).}
\begin{equation}\label{tst}
\psl {Lg}g\geq \mu \noh g^2 -\frac 1{\mu} \left( \psh g{\Lambda W}^2 + \sum_{j=1}^5 \psh g{\partial_{x_j} W}^2+\psh g{W}^2\right)
\end{equation}
\emph{(iii) Coercivity with  $Y$ orthogonality.} 
\begin{equation}\label{tstY}
\psl {Lg}g\geq \mu  \noh g^2 -\frac 1{\mu} \left( \psh g{\Lambda W}^2 + \sum_{j=1}^5 \psh g{\partial_{x_j} W}^2+\psl g{Y}^2\right)
\end{equation}
\emph{(iv) Localized coercivity.} For $\alpha>0$ small enough,
\begin{align}
& \int |\nabla g|^2 \varphi^2- f'(W) g^2 
 \geq   \mu  \int |\nabla g|^2 \varphi^2 
-\frac 1{\mu} \left( \psh g{\Lambda W}^2 + \sum_{j=1}^5 \psh g{\partial_{x_j} W}^2+\psl g{Y}^2\right)
\label{stloc}\end{align}
\end{lemma}
\begin{proof} (i) contains  well-known facts on $L$ that are easily checked directly.
We refer to Appendix D of \cite{OR} for the proof of \eqref{tst}. The proof of (iii) is standard since $(LY,Y)<0$.

Proof of \eqref{stloc}.
By direct computations 
$$
\int |\nabla (g\varphi)|^2 = \int |\nabla g|^2 \varphi^2 - \int |g|^2 \varphi \Delta \varphi.
$$
Note that (here the space dimension is $5$)
$$
\Delta \varphi= - 2\alpha\left( (3-2 \alpha) |x|^2 + 5 \right) \frac {\varphi}{(1+|x|^2)^2},
$$
and thus
$|\Delta \varphi | \leq 10  \alpha\frac {\varphi}{\langle x\rangle^2}$, and thus by \eqref{z2},
$$
    \int |g|^2 \varphi |\Delta \varphi| \leq 10 \alpha  \int |g|^2 \frac {\varphi^2}{\langle x\rangle^2}\leq 
    \delta( \alpha )  \int |\nabla (g\varphi)|^2,
$$
where $\delta(\alpha)\to 0$ as $\alpha\to 0$.
This implies the following estimate
\begin{equation}\label{eq:gd}
\left|  \int |\nabla g|^2 \varphi^2  -\int |\nabla (g\varphi)|^2 \right|  \leq   \delta(\alpha ) \int |\nabla (g\varphi)|^2.
\end{equation}

We check that
\begin{equation}\label{eq:qo}
 | \psh {g (1-\varphi)}{ \Lambda W}|
+   |\psh{g (1-\varphi}{\partial_{x_j}W}|+| \psl {g(1-\varphi)}{Y}|  \leq \delta (\alpha) \noh{g\varphi}. \end{equation}
Indeed, by the Cauchy-Schwarz inequality, the decay properties of  $W$  and Hardy inequality,
\begin{align*}
&\psh {g (1-\varphi)}{\Lambda W}^2 = \psl{g (1-\varphi)}{\Delta(\Lambda W)}^2  \\
&\leq \int \frac {(g \varphi)^2}{\langle x\rangle^2} \int |\Delta (\Lambda W)|^2 | 1-\varphi|^2 \frac {\langle x\rangle^2}{\varphi^2}
\leq \delta(\alpha) \noh {g\varphi}^2 ;
\end{align*}
the rest of the proof of \eqref{eq:qo} is similar.
We also have
\begin{equation}
\label{eq:qn}
\int W^{\frac 43} g^2 (1-\varphi^2)\leq  \left\|\frac {1-\varphi^2 }{\varphi^2} \langle x \rangle^2 W^{\frac 43}\right\|_{L^\infty}\int \frac {(g \varphi)^2}{\langle x\rangle^2}
 \lesssim \delta(\alpha) \noh {g\varphi}^2 .
\end{equation}

By \eqref{tstY}  applied to $g\varphi$ and then \eqref{eq:qo}, for $\alpha$ small,
\begin{align*}
  \psl {L (g\varphi)}{g\varphi} & \geq  \mu \noh {g \varphi}^2 
  - \frac 1{\mu}  \left( 
  \psh {g\varphi}{\Lambda W}^2 + \sum_{j=1}^5 \psh{g\varphi}{\partial_{x_j} W}^2+ \psh {g\varphi}{ W}^2\right) 
\\ & 
\geq \left( \mu - \delta(\alpha)\right)  \noh {g \varphi}^2 
  - \frac 1{\mu}  \left(  
  \psh g{\Lambda W}^2 + \sum_{j=1}^5 \psh{g}{\partial_{x_j} W}^2
+ \psh gW^2\right)
 \end{align*}
 Finally, using \eqref{eq:gd} and \eqref{eq:qn} we get \eqref{stloc}, for $\alpha$ small enough.
\end{proof}

\subsection{Energy linearization around $W_\ell$}

For $-1<\ell<1$,  let
\begin{equation}\label{eqWb}
W_{\ell }(x) = W\left(\frac {x_1}{\sqrt{1-\ell^2}}, \overline x\right),\quad 
(1-\ell^2) \partial_{x_1}^2 W_{\ell} + \overline \Delta W_{\ell} + W_\ell^{\frac 73}=0,
\end{equation}
so that $u(t,x) = W_{\ell }\left(x_1 - \ell t,\bar x\right)$ is a solution of \eqref{wave}. 
Note that
\begin{equation}\label{EWb}
E(W_\ell,-\ell  \partial_{x_1} W_\ell)-\ell^2\int|\partial_{x_1}W_\ell|^2
 =   (1-\ell^2)^{\frac 12} E(W,0).
\end{equation}
Let 
\begin{align}
& 
L_{\ell}   = -(1-\ell^2) \partial_{x_1}^2  -\overline \Delta   - f'(W_\ell) ,\\
&\psl{L_\ell g}g =   (1-\ell^2) \int |\partial_{x_1} g|^2+ \int \left(|\overline \nabla g|^2 -f'(W_\ell) g^2\right)
,\\
& H_\ell = \left(\begin{array}{cc}  -\Delta -f'(W_\ell) & -\ell \partial_{x_1} \\ \ell \partial_{x_1}& {\rm Id}\end{array}\right),\quad
\psl {H_\ell \fl g}{\fl g}   = 
\psl{L_\ell g}g + \| \ell \partial_{x_1} g + h\|_{L^2}^2.
\label{Hbeta}
\end{align}
As before, $L_\ell$ and $H_\ell$ are related to the linearization of the energy around $W_\ell$. Indeed,
proceeding as in \eqref{enerlin}, \begin{align*}
  & E(W_\ell + g, -\ell \partial_{x_1} W_\ell+h) + \ell \int \partial_{x_1}(W_\ell+g) (-\ell \partial_{x_1} W_\ell + h) \\
  & = E(W_\ell  , -\ell \partial_{x_1} W_\ell ) - \ell^2 \int (\partial_{x_1}W_\ell )^2\\
  & - \int (\Delta W_\ell) g - \int f(W_\ell) g- \ell \int (\partial_{x_1}W_\ell) h + \ell^2 \int (\partial_{x_1}^2 W_\ell)g  + \ell \int (\partial_{x_1}W_\ell) h\\
  & + \frac 12 \int |h|^2 + \frac 12  \int \left(|\nabla g|^2 - f'(W_\ell) g^2\right)
  + \ell \int h \partial_{x_1} g + O(\noh g^3).\end{align*}
and thus, using \eqref{eqWb} and \eqref{EWb},
\begin{align*}
  & E(W_\ell + g, -\ell \partial_{x_1} W_\ell+h) + \ell \int \partial_{x_1}(W_\ell+g) (-\ell \partial_{x_1} W_\ell + h) \\
  & =  (1-\ell^2)^{\frac 12} E(W  , 0 ) +\frac 12 \psl {H_\ell \fl g}{\fl g}   + O(\noh g^3).
\end{align*}
The following functions appear when studying the properties of the operators $H_\ell$ and $H_\ell J$
\begin{align*}
\fl Z_\ell^\Lambda = \left(\begin{array}{c} \Lambda W_\ell \\ - \ell \partial_{x_1}\Lambda W_\ell\end{array}\right),\quad
\fl Z_\ell^{\nabla_j} = \left(\begin{array}{c} \partial_{x_j} W_\ell \\ - \ell \partial_{x_1}\partial_{x_j} W_\ell\end{array}\right),\quad
\fl Z_\ell^W = \left(\begin{array}{c} W_\ell \\ - \ell \partial_{x_1}W_\ell \end{array}\right),
\end{align*}
\begin{align*}
Y_{\ell}(x) = Y\left(\frac {x_1}{\sqrt{1-\ell^2}},\overline x\right),
\quad 
\fl Z_\ell^\pm = \left(\begin{array}{c} \left(\ell \partial_{x_1} Y_\ell 
\pm \frac {\sqrt{\lambda_0}}{\sqrt{1-\ell^2}} Y_\ell\right) e^{\pm \frac {\ell \sqrt{\lambda_0}}{\sqrt{1-\ell^2}}x_1} \\ 
Y_\ell  e^{\pm \frac {\ell \sqrt{\lambda_0}}{\sqrt{1-\ell^2}}x_1 } \end{array}\right).
\end{align*}
We gather below several technical facts.
\begin{claim} The following hold for any $-1<\ell<1$,\\
\emph{(i) Properties of $L_\ell$.}
\begin{align}
& L_\ell (\Lambda W_\ell) = L_\ell (\partial_{x_j} W_\ell)=0,\quad   L_\ell Y_\ell = -\lambda_0 Y_\ell,\quad 
L_\ell W_\ell = - \frac 43 W_\ell^{\frac 73},\label{LW}
\end{align}
\emph{(ii) Properties of $H_\ell$ and $H_\ell J$.}
 \begin{equation}\label{ZW}
  H_\ell   \fl Z_\ell^\Lambda=  H_\ell   \fl Z_\ell^{\nabla_j}=0,\quad
  H_\ell \fl Z_\ell^W= - \frac 43 \left(\begin{array}{c} W_\ell^{\frac 73} \\ 0 \end{array}\right),
 \end{equation} 
 \begin{equation}\label{Zpm}
\psl {H_\ell \fl Z_\ell^W}{\fl Z_\ell^W} = -\frac 43 \int W_\ell^{\frac {10}3},
\quad
  - H_\ell J (\fl Z_\ell^\pm) = \pm \sqrt{\lambda_0} (1-\ell^2)^{\frac 12} \fl Z_\ell^\pm,
  \end{equation}
\begin{equation}
\label{oZ} 
\pse {\fl Z_\ell^\Lambda}{\fl Z_\ell^W}=
\pse {\fl Z_\ell^{\nabla_j}}{\fl Z_\ell^W}=0,
\quad \psl { \fl Z_\ell^\Lambda}{\fl Z_\ell^\pm}=
\psl { \fl Z_\ell^{\nabla_j}}{\fl Z_\ell^\pm}=0.
\end{equation} 
\emph{(iii) Antecedents.} 
There exist $\fl z_\ell^\pm$ such that
\begin{equation}
\label{eq:pz}
H_\ell \fl z_\ell^\pm = \fl Z_\ell^\pm, \quad
\psl {H_\ell \fl z_\ell^\pm}{\fl z_\ell^\pm}=0,\quad 
\pse { \fl z_\ell^\pm}{\fl Z_\ell^\Lambda} =
\pse { \fl z_\ell^\pm}{\fl Z_\ell^{\nabla_j}}=0 
\end{equation}
\end{claim}
\begin{proof}
The proof of \eqref{LW} follows from the same properties at $\ell=0$.

Next, note  that for any function $g$,  
\begin{equation}  \label{HZ}
  H_\ell \left(\begin{array}{c} g\\ -\ell \partial_{x_1} g \end{array}\right) = \left(\begin{array}{c} L_\ell g \\ 0 \end{array}\right),\quad
\psl{H_\ell \left(\begin{array}{c} g\\ -\ell \partial_{x_1} g\end{array}\right)}{\left(\begin{array}{c} g \\ -\ell \partial_{x_1} g \end{array}\right)} = \psl{L_\ell g}{g}.
\end{equation}

\emph{Proof of \eqref{ZW}.} First, by \eqref{HZ} and \eqref{LW}, $  H_\ell   (\fl Z_\ell^\Lambda)=    H_\ell   (\fl Z_\ell^{\nabla_j})=0$.
The identity concerning $\fl Z_\ell^W$ also follows directly from \eqref{HZ} and \eqref{LW}.
\medskip

\emph{Proof of \eqref{Zpm}.} 
Note that
$$
-H_\ell J =   \left(\begin{array}{cc} -\ell \partial_{x_1}  & \Delta + \frac 73 W_\ell^{\frac 43} \\ {\rm Id}& - \ell \partial_{x_1}\end{array}\right).
$$
On the one hand,
\begin{align*}
&- \ell  \partial_{x_1}\left( \left(\ell \partial_{x_1} Y_\ell 
\pm \frac {\sqrt{\lambda_0}}{\sqrt{1-\ell^2}} Y_\ell\right) e^{\pm \frac {\ell \sqrt{\lambda_0}}{\sqrt{1-\ell^2}}x_1}\right)   +\Delta\left(Y_\ell  e^{\pm \frac {\ell \sqrt{\lambda_0}}{\sqrt{1-\ell^2}}x_1 }\right) +\frac 73 W_\ell^{\frac 43}Y_\ell  e^{\pm \frac {\ell \sqrt{\lambda_0}}{\sqrt{1-\ell^2}}x_1}\\
& = -(L_\ell Y_\ell) e^{\pm \frac {\ell \sqrt{\lambda_0}}{\sqrt{1-\ell^2}}x_1}
\pm \frac{(1-\ell^2)\ell\sqrt{\lambda_0}}{\sqrt{1-\ell^2}} (\partial_{x_1} Y_\ell) e^{\pm \frac {\ell \sqrt{\lambda_0}}{\sqrt{1-\ell^2}}x_1}\\
&= \pm \sqrt{\lambda_0}(1-\ell^2)^{\frac 12} \left( \pm\sqrt{\lambda_0} (1-\ell^2)^{-\frac 12}Y_\ell  +{\ell}  (\partial_{x_1} Y_\ell)\right) e^{\pm \frac {\ell \sqrt{\lambda_0}}{\sqrt{1-\ell^2}}x_1}.\end{align*}
On the other hand,
\begin{align*}
&\left(\ell \partial_{x_1} Y_\ell 
\pm \frac {\sqrt{\lambda_0}}{\sqrt{1-\ell^2}}Y_\ell\right) e^{\pm \frac {\ell \sqrt{\lambda_0}}{\sqrt{1-\ell^2}}x_1} -\ell \partial_{x_1} \left(Y_\ell  e^{\pm \frac {\ell \sqrt{\lambda_0}}{\sqrt{1-\ell^2}}x_1 } \right)
 =\pm \sqrt{\lambda_0}(1-\ell^2)^{\frac 12} Y_\ell  e^{\pm \frac {\ell \sqrt{\lambda_0}}{\sqrt{1-\ell^2}}x_1 }\end{align*}
Thus,
$
- H_\ell J (\fl Z_\ell^\pm) = \pm \sqrt{\lambda_0}(1-\ell^2)^{\frac 12}\fl Z_\ell^\pm$.

\medskip

\emph{Proof of \eqref{oZ}.} 
Since $\psl{\partial_{x_j} \Lambda W}{\partial_{x_j} W}=0$
($\dot H^ 1$ scaling) and  $\psl{\partial_{x_j}  \partial_{x_{j'}}{W}}{\partial_{x_j} W}=0$
 (by parity), we have 
$\pse {\fl Z_\ell^\Lambda}{\fl Z_\ell^W}=
\pse {\fl Z_\ell^{\nabla_k}}{\fl Z_\ell^W}=0$.
Next, from \eqref{oZ}, the fact that $H_\ell$ is self-adjoint in $L^2$ and \eqref{ZW}, we have
$$\mp \sqrt{\lambda_0}(1-\ell^2)^{\frac 12}
\psl { \fl Z_\ell^\Lambda}{\fl Z_\ell^\pm}=
\psl { \fl Z_\ell^\Lambda}{ H_\ell J (\fl Z_\ell^\pm) }
=\psl {H_\ell  \fl Z_\ell^\Lambda}{ J (\fl Z_\ell^\pm) }=0.$$
The identity 
$\psl { \fl Z_\ell^{\nabla_j}}{\fl Z_\ell^\pm}=0$ is proved in a similar way.

\medskip

\emph{Proof of \eqref{eq:pz}.}
We set 
$$
\fl z_\ell^\pm = \mp\frac {J \fl Z_\ell^\pm}{\sqrt{\lambda_0}(1-\ell^2)^{1/2}}  
+ \alpha^{\Lambda,\pm} \fl Z_{\ell}^\Lambda
+\sum_{j=1}^5 \alpha^{\nabla_j,\pm}  \fl Z_{\ell}^{\nabla_j},
$$
where $\alpha^{\Lambda,\pm}$ and $\alpha^{\nabla_j,\pm} $ are chosen so that
$$
\pse { \fl z_\ell^\pm}{\fl Z_\ell^\Lambda} =
\pse { \fl z_\ell^\pm}{\fl Z_\ell^{\nabla_j}}=0 .
$$
By \eqref{ZW} and \eqref{Zpm}, we have $H_\ell \fl z_\ell^\pm = \fl Z_\ell^\pm$.
Finally, $H_\ell$ being self-adjoint, we have
$$
\psl {H_\ell \fl z_\ell^\pm}{\fl z_\ell^\pm}=\mp
\psl {H_\ell \fl z_\ell^\pm}{-\frac {J \fl Z_\ell^\pm}{\sqrt{\lambda_0}(1-\ell^2)}}=
\mp\frac 1{\sqrt{\lambda_0}(1-\ell^2)^{1/2}}\psl {\fl Z_\ell^\pm}{J \fl Z_\ell^\pm}=0.
$$
\end{proof}

We claim the following coercivity results with $\fl Z_k^{\pm}$ orthogonalities.

\begin{lemma}\label{pr:22}
Let $-1<\ell<1$. There exists $\mu>0$ such that,  for all $\fl g \in \dot H^1\times L^2$, the following holds.\\
 \emph{\rm (i)  Coercivity of $H_\ell$ with $Z_{\ell}^{\pm}$ orthogonalities.} 
\begin{align}
\psl {H_\ell \fl g}{\fl g} &\geq \mu  \|\fl g\|_E^2 
- \frac 1{\mu}\left( \pshb{g}{\Lambda W_\ell}^2 +\sum_{j=1}^5  \pshb {g}{\partial_{x_j}W_\ell}^2
+ \psl {\fl g}{\fl Z_\ell^{+}}^2 + \psl {\fl g}{\fl Z_\ell^{-}}^2\right)\label{eq:22}.\end{align}
\emph{\rm (ii) Localized coercivity.}  For $\alpha>0$ small enough,
\begin{align}
&    \int \left(|  \nabla g|^2 \varphi^2 -f'(W_\ell) g^2  +    h^2 \varphi^2
 + 2 \ell   (\partial_{x_1} g) h   \varphi^2 \right)
\nonumber \\ &\geq \mu  \int  \left(|\nabla g|^2   +   h^2 \right)\varphi^2  
  - \frac 1{\mu }\left( \psh{g}{\Lambda W_\ell}^2
+\sum_{j=1}^5  \psh {g}{\partial_{x_j} W_\ell}^2
+ \psl {\fl g}{\fl Z_\ell^{+}}^2 + \psl {\fl g}{\fl Z_\ell^{-}}^2\right)\label{eq:2.30}.\end{align}
\end{lemma}
 
\begin{proof}
 \emph{Proof of \eqref{eq:22}.} 
By a standard argument, it is equivalent to prove
\begin{align}
 \pshb{g}{\Lambda W_\ell}= \pshb {g}{\partial_{x_j}W_\ell}=\psl {\fl g}{\fl Z_\ell^{\pm}}=0
 \quad \Rightarrow \quad 
 \psl {H_\ell \fl g}{\fl g} \geq \mu  \|\fl g\|_E^2.\label{eq:22b}
\end{align}
Note that the proof of \eqref{eq:22b} is largely inspired by Proposition 2 in \cite{CMkg},
Lemma 5.1 in \cite{DMnls}, and  Proposition 5.5 in \cite{DMwave}.

\medskip

\emph{Case $\ell=0$.} Note that in this case
$ \fl Z_0^{\pm} = \left(\begin{array}{c}\pm \sqrt{\lambda_0} Y \\ Y\end{array}\right)$, and  $g$ as in \eqref{eq:22b} thus satisfies  the orthogonality conditions 
$ \psh g {\Lambda W} =\psh g { \partial_{x_j} W} =  \psl g{Y}=0$. 
Then, \eqref{eq:22b} follows from \eqref{tstY}.

\medskip

\emph{Case $\ell\neq0$.}
Note that \eqref{eq:22} is thus equivalent to
\begin{align}\label{2.11}
\pshb { g}{\Lambda W_\ell }=
\pshb { g}{\partial_{x_j}W_\ell }=  \psl {H_{\ell} \fl g}{\fl z_\ell^{\pm}}
=0\quad \Rightarrow\quad
\psl {H_\ell \fl g}{\fl g}  \geq \mu_\ell \|\fl g\|_E^2.
\end{align}
We decompose $g$ and $\fl z_\ell^\pm$ as follows
\begin{equation}\label{o3}
\fl g = a \fl Z_\ell^W + \fl g^\perp,\quad
\fl z_\ell^\pm = a^{\pm} \fl Z_\ell^W + \fl z_{\ell}^{\pm,\perp},
\quad 
\fl g^\perp = \left( \begin{array}{c}
                      g^\perp \\h^\perp
                     \end{array}
\right),\quad
\fl z_\ell^{\pm\perp} = \left( \begin{array}{c}
                       z_{\ell,1}^{\pm\perp} \\  z_{\ell,2}^{\pm\perp}
                     \end{array}
\right),
\end{equation}
where $a$ and $a^{\pm}$ are chosen so that
\begin{equation}\label{o1}
\pshb {  g^\perp}{W_\ell}=
\pshb {  z_{\ell,1}^{\pm,\perp}}{W_\ell}=0.
\end{equation}
We still have
$$
\pshb { g^\perp}{\Lambda W_\ell }=
\pshb { g^\perp}{\partial_{x_j}W_\ell }=0, \quad
\pshb {\fl z_{\ell,1}^{\pm,\perp}}{\Lambda W_\ell }=
\pshb {\fl z_{\ell,1}^{\pm,\perp}}{\partial_{x_j}W_\ell }=0.
$$
Note that since (see \eqref{ZW} and \eqref{eqWb})
$$H_\ell {\fl Z_\ell^W} =
- \frac 43 \left(\begin{array}{c} W_\ell^{\frac 73} \\0 \end{array}\right)
= \frac 43 \left(\begin{array}{c} (1-\ell^2) \partial_{x_1}^2 W_\ell +\overline\Delta W_\ell\\0 \end{array}\right),
$$
\eqref{o1} is equivalent to
\begin{equation}\label{o2}
\psl {H_\ell \fl g^\perp}{\fl Z_\ell^W}=
\psl {H_\ell \fl z_{\ell}^{\pm,\perp}}{\fl Z_\ell^W}=0.
\end{equation}

The decompositions \eqref{o3}   being orthogonal  with respect to $\psl {H_\ell.}.$, we have
\begin{align}
 \psl {H_\ell \fl g}{\fl g} &=  a^2 \psl {H_\ell\fl Z_\ell^W}{\fl Z_\ell^W}
 + \psl {H_\ell\fl g^\perp}{\fl g^\perp},\nonumber\\
 0 =  \psl {H_\ell \fl z_\ell^\pm }{\fl z_\ell^\pm} &=  (a^\pm)^2 \psl {H_\ell\fl Z_\ell^W}{\fl Z_\ell^W}
 + \psl {H_\ell  \fl z_{\ell}^{\pm,\perp}}{ \fl z_{\ell}^{\pm,\perp}},\nonumber\\
 0 = \psl {H_{\ell} \fl g}{\fl z_\ell^{\pm}} & = a a^\pm \psl {H_\ell\fl Z_\ell^W}{\fl Z_\ell^W}  + \psl {H_\ell  \fl g^\perp}{ \fl z_{\ell}^{\pm,\perp}},\label{pisc}
\end{align}
which imply (recall that $\psl {H_\ell\fl Z_\ell^W}{\fl Z_\ell^W}<0$), from \eqref{Zpm},
\begin{equation}
\label{fin1}
\psl {H_\ell \fl g}{\fl g}  = 
- \frac {\psl {H_\ell  \fl g^\perp}{ \fl z_{\ell}^{-,\perp}}\psl {H_\ell  \fl g^\perp}{ \fl z_{\ell}^{+,\perp}}}{\sqrt{\psl {H_\ell \fl z_{\ell}^{-,\perp} }{\fl z_{\ell}^{-,\perp}}\psl {H_\ell \fl z_{\ell}^{+,\perp} }{\fl z_{\ell}^{+,\perp}}}}
 + \psl {H_\ell\fl g^\perp}{\fl g^\perp}.
\end{equation}
Let
$$
A=\sup_{\fl\omega\in{\rm Span}(\fl z_\ell^{+,\perp},\fl z_\ell^{-,\perp})}
\left| \frac {\psl {H_\ell  \fl \omega}{ \fl z_{\ell}^{-,\perp}} }
{\sqrt{\psl {H_\ell \fl z_{\ell}^{-,\perp} }{\fl z_{\ell}^{-,\perp}}\psl {H_\ell\fl  \omega}{\fl  \omega}}}
\frac {\psl {H_\ell  \fl  \omega}{ \fl z_{\ell}^{+,\perp}} }
{\sqrt{\psl {H_\ell \fl z_{\ell}^{+,\perp} }{\fl z_{\ell}^{+,\perp}}\psl {H_\ell\fl  \omega}{\fl  \omega}}}
  \right|
$$
Since $(H_\ell.,.)$ is positive definite on ${\rm Span}(\Delta W_\ell,\Delta \Lambda W_\ell,\Delta \partial_{x_j} W_\ell)^\perp$, applying Cauchy-Schwarz inequality to each of the term of the product above, we find $A\leq 1$. Moreover, $A=1$ would imply that $ \fl z_{\ell}^{-,\perp}$ and $ \fl z_{\ell}^{+,\perp}$ are proportional, which is clearly not true for $\ell\neq 0$ (for example, due to different behavior at $\infty$ of $\fl Z_k^{\pm}$). Thus, $A<1$. As a consequence, we also obtain  that for all 
$\fl \omega \in  {\rm Span}(\Delta W_\ell,\Delta \Lambda W_\ell,\Delta \partial_{x_j} W_\ell)^\perp$,
$$
\left| \frac {\psl {H_\ell  \fl \omega}{ \fl z_{\ell}^{-,\perp}} }
{\sqrt{\psl {H_\ell \fl z_{\ell}^{-,\perp} }{\fl z_{\ell}^{-,\perp}}\psl {H_\ell\fl  \omega}{\fl  \omega}}}
\frac {\psl {H_\ell  \fl  \omega}{ \fl z_{\ell}^{+,\perp}} }
{\sqrt{\psl {H_\ell \fl z_{\ell}^{+,\perp} }{\fl z_{\ell}^{+,\perp}}\psl {H_\ell\fl  \omega}{\fl  \omega}}}
  \right|\leq A (H_\ell \fl\omega,\fl\omega)_{L^2}.
$$
Thus, by \eqref{fin1} and then \eqref{tst} (after change of variables),
$$
\psl {H_\ell \fl g}{\fl g}   \geq (1-A) \psl {H_\ell\fl g^\perp}{\fl g^\perp}
\geq c \noe{\fl g^\perp}^2.
$$
The result then follows from $|a|\lesssim \|\fl g^\perp\|_E$ from \eqref{pisc}.

\medskip

Proof of \eqref{eq:2.30}. First, we apply \eqref{eq:22} on $\fl g \varphi$:
\begin{align*}
& \psl {H_\ell (\fl g \varphi)}{\fl g \varphi}  \geq \mu  \noe{\fl g \varphi}^2 \\ &
- \frac 1{\mu}\left( ({\fl g \varphi},{\Lambda W_\ell})_{\dot H^1_\ell}^2
+\sum_{j=1}^5 ({\fl g \varphi},{\partial_{x_j} W_\ell})_{\dot H^1_\ell}^2
+ \psl {\fl g \varphi}{\fl Z_\ell^{+}}^2 + \psl {\fl g \varphi}{\fl Z_\ell^{-}}^2\right).
\end{align*}
Recall that
\begin{align*}
 \psl {H_\ell (\fl g \varphi)}{\fl g \varphi} & =\int |\nabla (g \varphi)|^2   - \frac 73 \int W_\ell^{\frac 43} g^2 \varphi^2 
 + 2 \ell \int \partial_{x_1} (g \varphi) (h \varphi)
 +\int h^2 \varphi^2.
\end{align*}
Note that $\partial_{x_1} \varphi = \frac {-2 \alpha x_1}{1+|x|^2} \varphi$ and so
\begin{align*}
&\left| \int  \partial_{x_1} (g \varphi) (g \varphi)
- \int  (\partial_{x_1}  g)  h \varphi^2 \right|
= \left| \int gh (\partial_{x_1} \varphi) \varphi\right| \leq
 C \alpha \int |g| |h| \frac {\varphi^2}{\langle x\rangle} 
 \\& \leq C \alpha \left( \int \frac {(g\varphi)^2}{\langle x\rangle^2}\right)^{\frac 12} \left( \int |h|^2 \varphi^2\right)^{\frac 12} \leq C \alpha \int |\nabla (g\varphi)|^2 + C \alpha \int |h|^2 \varphi^2.
\end{align*}
Thus, using \eqref{eq:gd},
\begin{align*}
&\left| \psl {H_\ell (\fl g \varphi)}{\fl g \varphi}
 -  \int \left(|  \nabla g|^2  -   f'(W_\ell) g^2  +    h^2  
 + 2 \ell   (\partial_{x_1} g) h  \right)  \varphi^2\right| \leq \delta(\alpha) \noe {\fl g \varphi}^2.
\end{align*}
To complete the proof, we just notice that as in \eqref{eq:qo}
\begin{equation}\label{eq:qo2}
\left| \psl {\fl g (1-\varphi)}{\fl Z_\ell^{\pm}}\right|
 \leq \delta (\alpha) \noe{\fl g\varphi},
 \end{equation}
 and similarly for the other scalar products appearing in \eqref{eq:2.30}, and as
 in \eqref{eq:qn},
 \begin{equation}
\label{eq:qnb}
\int W_\ell^{\frac 43} g^2 (1-\varphi^2) \lesssim \delta(\alpha) \noh {g\varphi}^2 .
\end{equation}
Combining these estimates, we obtain \eqref{eq:2.30}, for $\alpha$ small enough.
\end{proof}

\subsection{Energy linearization around $W_\bell$}
We only define some  notation generalizing the previous section.
For $\bell \in \R^5$ such that 
 $|\boldsymbol{\ell}|< 1$, $W_{\bell}$ defined in \eqref{defWbb} solves
\begin{equation}\label{eqWbb}
  \Delta W_\bell-\bell\cdot\nabla(\bell\cdot\nabla W_\bell)+W_\bell^{\frac 73}=0.
  \end{equation}
The following operators are related to the linearization of the energy around $W_\bell$
$$ L_{\bell}   = -\Delta -\bell\cdot\nabla(\bell\cdot\nabla)  - f'(W_\bell)  ,\quad H_\bell = \left(\begin{array}{cc}  -\Delta - f'(W_\bell) & -\bell\cdot \nabla \\ \bell \cdot \nabla & {\rm Id}\end{array}\right).$$
Set  
\begin{align*}
\fl Z_{\bell}^\Lambda = \left(\begin{array}{c} \Lambda W_{\bell} \\ - {\bell}\cdot \nabla(\Lambda W_{\bell})\end{array}\right),\quad
\fl Z_{\bell}^{\nabla_j} = \left(\begin{array}{c} \partial_{x_j} W_{\bell} \\ - {\bell}\cdot \nabla (\partial_{x} W_{\bell})\end{array}\right),\quad
\fl Z_{\bell}^W = \left(\begin{array}{c} W_{\bell} \\ - {\bell}\cdot \nabla W_{\bell} \end{array}\right),
\end{align*}
$$
Y_{\bell}=Y\left(\left(\frac{1}{\sqrt{1-|\boldsymbol{\ell}|^2}}-1\right) \frac{\boldsymbol{\ell}(\boldsymbol{\ell}\cdot x)}{|\boldsymbol\ell|^2} +x\right),
\quad 
 \fl Z_{\bell}^\pm = \left(\begin{array}{c} \left({\bell} \cdot \nabla  Y_{\bell} 
\pm \frac {\sqrt{\lambda_0}}{\sqrt{1-|\bell|^2}} Y_{\bell}\right) e^{\pm \frac {  \sqrt{\lambda_0}}{\sqrt{1-|{\bell}|^2}} {\bell} \cdot x} \\ 
Y_{\bell}  e^{\pm \frac { \sqrt{\lambda_0}}{\sqrt{1-|{\bell}|^2}} {\bell} \cdot  x } \end{array}\right).
$$Note from \eqref{oZ} and \eqref{Zpm},
\begin{equation}\label{oZbb}
\psl { \fl Z_\bell^\Lambda}{\fl Z_\bell^\pm}=
\psl { \fl Z_\bell^{\nabla_j}}{\fl Z_\bell^\pm}=0,
\end{equation}
\begin{equation}\label{Zbbeta}
	- H_\bell J \fl Z_\bell^{\pm} = \pm \sqrt{\lambda_0} (1-|\bell|^2)^{\frac 12} \fl Z_\bell^{\pm}.
\end{equation}
\section{Decomposition around the sum of $K$ solitons}

We prove in this section a general decomposition around $K$ solitons.
Let $K\geq 1$ and for any $k\in \{1,\ldots,K\}$,  let $\lambda^\infty_k>0$, $\mathbf{y}_k^\infty \in \R^5$,
 $\bell_k\in \R^5$, $|\bell_k|<1$ with $\bell_{k'}\neq \bell_k$ for $k'\neq k$.
 
First, for $\fl G= (G,H)$, set
$$
(\theta_k^\infty G)(t,x) = \frac {\iota_k}{(\lambda_k^\infty)^{3/2}} G\left(\frac {x -\bell_k t - \mathbf{y}_k^\infty}{\lambda_k^\infty}\right),
\quad 
 \fl \theta_k^\infty \fl  G 
= \left(
  \begin{array}{c}\theta_k^\infty G \\[.2cm] 
  \displaystyle \frac {\theta_k^\infty} {\lambda_k^\infty}  H \end{array}
  \right).
$$
In particular, set
$$
  W_k^\infty =\theta_k^\infty W_{\bell_k}, \quad
\fl W_k^\infty  = \left(\begin{array}{c} \theta_k^\infty W_{\bell_k} \\[.2cm]
 \displaystyle -\frac  {\bell_k} {\lambda_k^\infty} \cdot{\theta_k^\infty}  ( \nabla W_{\bell_k}) \end{array}\right).
$$

Second,  for $C^1$ functions $\lambda_k(t)>0$, $\mathbf{y}_k(t)\in \R^5$  to be chosen, let
\begin{equation}\label{thetak}
(\theta_k G)(t,x) = \frac {\iota_k}{\lambda_k^{3/2}(t)} G\left(\frac {x -\bell_k t - \mathbf{y}_k(t)}{\lambda_k(t)}\right),
\quad
\fl \theta_k \fl  G  = \left(\begin{array}{c}\theta_k G \\[.2cm] \displaystyle \frac {\theta_k} {\lambda_k }  H \end{array}\right),\quad
\fl {\tilde \theta}_k \fl G  = \left(\begin{array}{c}\displaystyle \frac {\theta_k}{\lambda_k} G\\[.4cm]  \theta_k H \end{array}\right).
\end{equation}
In particular, set
\begin{equation}\label{defWk}
W_k  = \theta_k W_{\bell_k} ,\quad
\fl W_k  = \left(\begin{array}{c} \theta_k W_{\bell_k} \\[.2cm] \displaystyle -\frac  {\bell_k} {\lambda_k} \cdot{\theta_k}  ( \nabla W_{\bell_k}) \end{array}\right).
\end{equation}
In what follows $\sum_{k=1}^K$ is often simply denoted  by $\sum_k$.

\begin{lemma}[Properties of the decomposition]\label{le:4} 
There exist $T_0\gg 1$ and $0<\delta_0\ll 1$ such that if $u(t)$ is a solution of \eqref{wave} on
$[T_1,T_2]$, where $T_0\leq T_1<T_2$,  such that
\begin{equation}\label{hyp:4}
	\forall t\in [T_1,T_2],\quad
	\left\|\fl u(t)-\sum_k \fl W_k^\infty(t)\right\|_{\dot H^1\times L^2}
	\leq \delta_0,
\end{equation}
then there exist $C^1$  functions $\lambda_k>0$, $\mathbf{y}_k$ on $[T_1,T_2]$ such that, 
$\fl \e(t)$ being defined by
\begin{equation}
\label{eps2}
 \fl \e=\left(\begin{array}{c}\e \\ \eta \end{array}\right),\quad 
\fl u = \left(\begin{array}{c} u \\u_t\end{array}\right) =
\sum_k\fl W_k +  \fl \e, 
\end{equation}
the following hold on $[T_1,T_2]$.\\
\emph{(i) First properties of the decomposition.}
\begin{equation}
\label{ortho}
\pshbbk {\e}{\theta_k (\Lambda W_{\bell_k})}=
\pshbbk {\e}{\theta_k (\partial_{x_j} W_{\bell_k})}= 0,
\end{equation}
\begin{equation}\label{bounds}
  |\lambda_k(t)-\lambda_k^{\infty}|
  +|\mathbf{y}_k(t)-\mathbf{y}_k^\infty|
  +\|\fl \e\|_{E}
  \lesssim \left\|\fl u(t)-\sum_k \fl W_k^\infty(t)\right\|_{\dot H^1\times L^2}
\end{equation}
\emph{(ii)  Equation of $\fl \e$.}
\begin{equation}\label{syst_e}\left\{\begin{aligned}
\e_t  & =  \eta  + {\rm Mod}_\e\\
	\eta_t & 
	=  \Delta \e +f\left(   \SWk   + \e\right) 
-   f\left(\SWk\right)   + R_{W} + {\rm Mod}_\eta
\end{aligned}\right.\end{equation} 
where
\begin{align}
R_{W} &=f\left(\SWk\right)   - \sum_k f(W_k),
\label{RW}\end{align}
\begin{align}
{\rm Mod}_\e &=  \sum_k\frac{\dot\lambda_k}{\lambda_k} \theta_k(\Lambda W_{\bell_k})
+   \sum_k \frac{{\dot {\mathbf{y}}}_k}{\lambda_k} \cdot \theta_k(\nabla W_{\bell_k})\label{Mod_e}\\
{\rm Mod}_\eta &= 
-  \sum_k \frac{\dot \lambda_k} {\lambda_k^2} \bell_k\cdot\theta_k(\nabla\Lambda W_{\bell_k}) 
-  \sum_k \frac{{\dot {\mathbf{y}}}_k} {\lambda_k^2} \cdot \theta_k(\nabla ( \bell_k \cdot \nabla W_{\bell_k})).
\label{Mod_eta}
\end{align}
\emph{(iii)  Parameters equations.}
\begin{equation}
\label{le:p}
\sum_k |\dot \lambda_k(t)|+|\dot {\mathbf{y}}_k(t)|
\lesssim \noe {\fl\e(t)}  .
\end{equation}
\emph{(iv) Unstable directions.} Let
\begin{equation}
\label{le:dz}
z_k^{\pm} (t)  = \psl {\fl \e(t)}{\fl {\tilde\theta}_k \fl Z_{\bell_k}^{\pm}}.
\end{equation}
Then,
\begin{equation}
\label{le:z}
\left| \frac d{dt} z_k^{\pm}(t) \mp \frac{\sqrt{\lambda_0}}{\lambda_k}(1-|\bell_k|^2)^{\frac 12}z_k^{\pm} (t)\right|
\lesssim \noe {\fl\e(t)}^2 + \frac {\noe {\fl\e(t)}}t + \frac 1 {t^3}.
\end{equation}
\end{lemma}
\begin{proof}
{\bf Step 1.} Decomposition. Let $T_0\gg 1$, fix $t\geq T_0$ and assume that \eqref{hyp:4} holds for $t$.
Let 
$$
\Gamma^\infty=(\lambda_k^\infty,\mathbf{y}_k^\infty)_{k\in \{1,\ldots,K\}},\quad
\Gamma=(\lambda_k,\mathbf{y}_k)_{k\in \{1,\ldots,K\}}\in ((0,+\infty)\times\R^5)^K,
$$
where $\lambda_k$ and $\mathbf{y}_k$ are to be found (depending on $t$).
Consider the map
\begin{align*}
\Phi : \dot H^1 & \times ((0,+\infty)\times\R^5)^K \to \R^{6K}\\
(\omega,\Gamma) & \mapsto \left(
(\omega+\sum_{k'} W_{k'}^\infty -\sum_{k'} \theta_{k'} W_{\ell_{k'}},\theta_k(\Lambda W_{\bell_k}))_{\dot H^1_{\bell_k}},\right.\\
& \qquad (\omega+\sum_{k'} W_{k'}^\infty -\sum_{k'} \theta_{k'} W_{\ell_{k'}},\theta_k(\partial_{x_1} W_{\bell_k}))_{\dot H^1_{\bell_k}},\\
& \left.\qquad \ldots,
(\omega+\sum_{k'} W_{k'}^\infty -\sum_{k'} \theta_{k'} W_{\ell_{k'}},\theta_k(\partial_{x_5} W_{\bell_k}))_{\dot H^1_{\bell_k}}
\right)_{k\in \{1,\ldots,K\}},
\end{align*}
where $\theta_k$ is defined in \eqref{thetak}.
By explicit computations, we have
\begin{align*}
&\left(d_\Gamma \Phi (0,\Gamma^\infty)\cdot\tilde \Gamma\right)_k 
\\
& =\left(
\sum_{k'} \frac{\tilde\lambda_{k'}}{\lambda_{k'}^\infty}
\left(\theta_{k'}^\infty(\Lambda W_{\bell_{k'}}),\theta_{k}^\infty(\Lambda W_{\bell_{k}})\right)_{\dot H^1_{\bell_k}}
+\sum_{k'} \frac{\tilde {\mathbf{y}}_{k'}}{\lambda_{k'}^\infty}\cdot
\left(\theta_{k'}^\infty(\nabla W_{\bell_{k'}}),\theta_{k}^\infty(\Lambda W_{\bell_{k}})\right)_{\dot H^1_{\bell_k}},
\right.\\
& \sum_{k'} \frac{\tilde\lambda_{k'}}{\lambda_{k'}^\infty}
\left(\theta_{k'}^\infty(\Lambda W_{\bell_{k'}}),\theta_{k}^\infty(\partial_{x_1} W_{\bell_{k}})\right)_{\dot H^1_{\bell_k}} +\sum_{k'} \frac{\tilde {\mathbf{y}}_{k'}}{\lambda_{k'}^\infty}\cdot
\left(\theta_{k'}^\infty(\nabla W_{\bell_{k'}}),\theta_{k}^\infty(\partial_{x_1}W_{\bell_{k}})\right)_{\dot H^1_{\bell_k}},
\ldots,
\\
&\sum_{k'} \frac{\tilde\lambda_{k'}}{\lambda_{k'}^\infty}
\left(\theta_{k'}^\infty(\Lambda W_{\bell_{k'}}),\theta_{k}^\infty(\partial_{x_5} W_{\bell_{k}})\right)_{\dot H^1_{\bell_k}}
+\left. \sum_{k'} \frac{\tilde {\mathbf{y}}_{k'}}{\lambda_{k'}^\infty}\cdot
\left(\theta_{k'}^\infty(\nabla W_{\bell_{k'}}),\theta_{k}^\infty(\partial_{x_5} W_{\bell_{k}})\right)_{\dot H^1_{\bell_k}},
 \right)
 \end{align*}
 Thus, by parity property, $\psh{\partial_{x_j} W}{\partial_{x_j'} W}=0$ and the decay properties of $W$,
 \begin{align*}
&\left(d_\Gamma \Phi (0,\Gamma^\infty)\cdot\tilde \Gamma\right)_k 
\\
 & = \left(
 \frac{\tilde\lambda_{k}}{\lambda_{k}^\infty}
\left(\theta_{k}^\infty(\Lambda W_{\bell_{k}}),\theta_{k}^\infty(\Lambda W_{\bell_{k}})\right)_{\dot H^1_{\bell_k}},
\right. \frac{\tilde {\mathbf{y}}_{{k},1}}{\lambda_{k}^\infty}
\left(\theta_{k}^\infty(\partial_{x_1} W_{\bell_{k}}),\theta_{k}^\infty(\partial_{x_1}W_{\bell_{k}})\right)_{\dot H^1_{\bell_k}},\ldots,
\\
&\left. \sum_{k}\frac{\tilde {\mathbf{y}}_{k,5}}{\lambda_{k}^\infty}
\left(\theta_{k}^\infty(\partial_{x_5} W_{\bell_{k}}),\theta_{k}^\infty(\partial_{x_5} W_{\bell_{k}})\right)_{\dot H^1_{\bell_k}} \right)+ \mathcal E \cdot \tilde \Gamma,
\end{align*}
where $\|\mathcal E\|\lesssim \frac 1{T_0}$.
Hence, $d_\Gamma \Phi (0,\Gamma^\infty)$ is invertible for $T_0$ large enough, with a  lower bound uniform  in  $\Gamma^\infty$. Moreover,
$\Phi (0,\Gamma^\infty)=0$. Therefore, by the implicit function theorem (in fact, a uniform variant of the IFT), there exist $0<\delta_1\ll 1$,
$0<\delta_2\ll 1$, and a continuous map
$$
\Psi : B_{\dot H^1}(0,\delta_1) \to B_{((0,+\infty)\times\R^5)^K} (\Gamma^\infty,\delta_2),
$$
such that for all $\omega \in B_{\dot H^1}(0,\delta_1)$ and all $\Gamma \in B_{((0,+\infty)\times\R^5)^K} (\Gamma^\infty,\delta_2)$,
$$
\Phi(\omega, \Gamma)=0 \quad \hbox{if and only if} \quad \Gamma = \Psi(\omega).
$$
Moreover,
$$
|\Psi(\omega)-\Gamma^\infty|\lesssim \|\omega\|_{\dot H^1}.
$$

This defines a continuous map $t\in [T_1,T_2]\mapsto(\lambda_k(t),\mathbf{y}_k(t))_{k\in \{1,\ldots,K\}}$ 
such that 
$$
 |\lambda_k(t)-\lambda_k^{\infty}|
  +|\mathbf{y}_k(t)-\mathbf{y}_k^\infty|
 \lesssim \left\|u(t)-\sum_k W_k^\infty(t)\right\|_{\dot H^1}
$$
and such that
$\fl \e(t)$ defined by \eqref{eps2} satisfies the orthogonality conditions \eqref{ortho}.
Since
\begin{equation}\label{diff}
\left\|W_k(t) -W_k^\infty(t)\right\|_{\dot H^1}\lesssim |\lambda_k(t)-\lambda_k^{\infty}|+|\mathbf{y}_k(t)-\mathbf{y}_k^\infty|,
\end{equation}
we have
\begin{align*}
\|\fl \e(t)\|_E 
  &\lesssim \sum_k \left\|W_k(t) -W_k^\infty(t)\right\|_{\dot H^1}
  +\left\|\fl u(t)-\sum_k \fl W_k^\infty(t)\right\|_{\dot H^1\times L^2}\\
      & \lesssim  \left\|\fl u(t)-\sum_k \fl W_k^\infty(t)\right\|_{\dot H^1\times L^2},
\end{align*}
and \eqref{bounds} is proved.

For future reference, note that  
\begin{equation}\label{diff2}
\left\|\langle x\rangle^{1/2}\nabla (W_k(t) -W_k^\infty(t))\right\|_{L^2}
\lesssim t^{\frac 12} \left(|\lambda_k(t)-\lambda_k^{\infty}|+|\mathbf{y}_k(t)-\mathbf{y}_k^\infty|\right).
\end{equation}
Thus, if  $(u(t),\partial_t u(t))\in Y^1\times Y^0$,  then we   have
\begin{align*}
&\|\langle x\rangle^{1/2} \nabla \e(t)\|_{L^2} +  \|\langle x\rangle^{1/2}\eta(t)\|_{L^2}
        \lesssim  \left\|\langle x\rangle^{1/2} \nabla \left(u(t) -  \sum_k W_k^\infty(t) \right) \right\|_{L^2}  \\
      &+ \left\|\langle x\rangle^{1/2} \left( \partial_t u (t)+ \sum_k (\bell_k \cdot \nabla W_k^\infty)(t)\right)\right\|_{L^2}
      +t^{\frac 12}\left\| \fl u(t) -  \sum_k \fl W_k^\infty(t)  \right\|_{H^1\times L^2}, 
\end{align*}
and also
\begin{equation}\label{diff4}
\|\fl \e(t)\|_{Y^1\times Y^0} 
  \lesssim  t^{\frac 12}\left\| \fl u(t) -  \sum_k \fl W_k^\infty(t)  \right\|_{Y^1\times Y^0}.\end{equation}

\medskip

{\bf Step 2.} Equation of $\fl \e$ and parameter estimates.
We formally derive the equations of $\fl \e(t)$, $\lambda_k(t)$ and $\mathbf{y}_k(t)$ from the equation of $u$.
First, 
\begin{align}
\e_t 	& = u_t - \sum_{k} \partial_t  W_k = \eta - \sum_{k}  \frac {\bell_k} {\lambda_k} \cdot \theta_k( \nabla W_{\bell_k})
		-\sum_{k}   \partial_t (\theta_k W_{\ell_k}) \nonumber \\
		 &   =  \eta + \sum_{k}   \frac{\dot \lambda_k }{\lambda_k} \theta_k(\Lambda W_{\bell_k})
+ \sum_{k}  \frac{\dot {\mathbf{y}}_k  }{\lambda_k} \cdot \theta_k(\nabla W_{\bell_k}),\label{eqe1}
\end{align}
since, by direct computations,
\begin{equation}\label{aeqe1}
 \partial_t (\theta_k W_{\bell_k}) =
-\frac {\bell_k}{\lambda_k} \cdot\theta_k(\nabla W_{\bell_k})
-  \frac{\dot \lambda_k }{\lambda_k} \theta_k(\Lambda W_{\bell_k})
-  \frac{\dot {\mathbf{y}}_k  }{\lambda_k} \cdot \theta_k(\nabla W_{\bell_k}).
\end{equation} 
Second (using \eqref{aL})
\begin{align*}
\eta_t   = u_{tt} + \partial_t \left(\sum_{k}   \frac {\bell_k} {\lambda_k} \cdot \theta_k( \nabla W_{\bell_k})\right)
&= \Delta u + |u|^{\frac 43} u 
-  \sum_{k}  \frac{\bell_k} {\lambda_k^2} \cdot \theta_k(\nabla ( \bell_k \cdot \nabla W_{\bell_k}))
\\
& 
-  \sum_{k} \frac{\dot \lambda_k } {\lambda_k^2} \bell_k\cdot\theta_k(\nabla\Lambda W_{\bell_k}) 
-  \sum_{k}  \frac{\dot {\mathbf{y}}_k} {\lambda_k^2} \cdot \theta_k(\nabla ( \bell_k \cdot \nabla W_{\bell_k})).
\end{align*}
Using $u = \sum_{k}  \theta_k W_{\bell_k} + \e$,
we have
\begin{equation}
\label{2.17}
\Delta u =\sum_{k}  \frac {\theta_k}{\lambda_k^2} (\Delta W_{\bell_k})   + \Delta \e ,
\end{equation}
and  
\begin{align}
|u|^{\frac 43} u =f(u)
& = \sum_{k}  f\left(\theta_k W_{\bell_k}\right) 
+   \left( \sum_{k}  f'(\theta_k W_{\bell_k}) \right) \e + R_{\rm NL} +R_W,
\label{2.18}
\end{align}
where $R_W$  is defined in \eqref{RW}
and
$$
R_{\rm NL}  = f\left(   \SWk   + \e\right) 
-   f\left(\SWk\right)     - f' \left(\sum_k  W_k\right) \e$$
Since
$$
f(\theta_k W_{\ell_k})= \frac {\theta_k}{\lambda_k^2} f(W_{\bell_k}),\quad
f'(\theta_k W_{\bell_k})  = \frac {\theta_k}{\lambda_k^{\frac 12}} 
f'(W_{\bell_k}),
$$
we obtain 
$$
\Delta u +|u|^{\frac 43} u=
 \sum_{k}  \frac {\theta_k}{\lambda_k^2} \left( \Delta W_{\bell_k} + W_{\bell_k}^{\frac 73}\right) +
\Delta \e 
+ \frac 73\left( \sum_{k}  \frac {\theta_k}{\lambda_k^{\frac 12}} W_{\bell_k}^{\frac 43}  \right) \e + R_{\rm NL}+R_W.
$$
Using \eqref{eqWbb}, 
we obtain  
\begin{align*}
\eta_t &  =  \Delta \e 
+ \frac 73\left(  \sum_{k}  \frac {\theta_k}{\lambda_k^{\frac 12}} W_{\bell_k}^{\frac 43}\right) \e + R_{\rm NL}+R_W
\\
& 
-  \sum_{k} \dot \lambda_k \frac{\bell_k} {\lambda_k^2} \cdot\theta_k(\nabla\Lambda W_{\bell_k}) 
-  \sum_{k}  \frac{\dot {\mathbf{y}}_k} {\lambda_k^2} \cdot \theta_k(\nabla ( \bell_k \cdot \nabla W_{\bell_k})).
\end{align*}
In conclusion for $\fl \e$, we obtain
\begin{equation}
\label{fle}
\fl \e_t = \fl {\mathcal L} \fl \e  +\fl {\rm Mod}+ \fl R_{\rm NL} +\fl R_W,
\end{equation}
where
\begin{equation}
\label{eq:L}
\fl {\mathcal L}  = \left(\begin{array}{cc}0 & 1 \\
\Delta + \frac 73\left( \sum_k \frac {\theta_k}{\lambda_k^{1/2}} W_{\bell_k}^{4/3}\right)  & 0\end{array}\right),\qquad 
\fl R_{\rm NL} = \left(\begin{array}{c}0 \\R_{\rm NL}\end{array}\right)
\qquad 
\fl R_{W} = \left(\begin{array}{c}0 \\ R_{W}\end{array}\right),
\end{equation}
and
\begin{equation}\label{eq:mod}
\fl {\rm Mod}=
 \sum_{k}  \frac {\dot \lambda_k }{\lambda_k} \fl\theta_k \fl Z_{\bell_k}^{\Lambda} 
+ \sum_{k}  \frac {\dot {\mathbf{y}}_k}{\lambda_k} \cdot \fl \theta_k \fl Z_{\bell_k}^{\nabla}.
\end{equation}

{\bf Step 3.}
Now, we derive the equations of $\lambda_k$ and $\mathbf{y}_k$ from  the orthogonality   \eqref{ortho}. First,
\begin{align*}
   \frac d{dt} \pshbbun {\e}{ \theta_1 (\Lambda W_{\bell_1}) } = 
\pshbbun{\e_t}{ \theta_1 (\Lambda W_{\bell_1}) }+
\pshbbun {\e}{\partial_t \left(\theta_1 (\Lambda W_{\bell_1})\right) }=0
 \end{align*}
 Thus, using \eqref{eqe1},
 \begin{align}
0= &   \pshbbun{\eta}{ \theta_1 (\Lambda W_{\bell_1}) }
 -\pshbbun{\e}{\frac{\bell_1}{\lambda_1}\cdot \theta_1(\nabla  (\Lambda W_{\bell_1}))}\nonumber\\
& +\frac{\dot\lambda_1}{\lambda_1}\left( \pshbbun{\theta_1 (\Lambda W_{\bell_1}) }{ \theta_1 (\Lambda W_{\bell_1}) } -\pshbbun {\e}{  \left(\theta_1 (\Lambda^2 W_{\bell_1})\right) }\right)  \nonumber\\
 & + \pshbbun{ \frac{\dot {\mathbf{y}}_1}{\lambda_1} \cdot\theta_1 (\nabla W_{\bell_1}) }{ \theta_1 (\Lambda W_{\bell_1}) } -\pshbbun {\e}{\frac{\dot {\mathbf{y}}_1}{\lambda_1} \cdot\theta_1 (\nabla\Lambda  W_{\bell_1})}  \nonumber \\
 & + \sum_{k=2}^K \left(\frac{\dot\lambda_k}{\lambda_k}  \pshbbun{\theta_k (\Lambda W_{\bell_k}) }{ \theta_1 (\Lambda W_{\bell_1}) }
 + \pshbbun{ \frac{\dot {\mathbf{y}}_k}{\lambda_k} \cdot\theta_1 (\nabla W_{\bell_1}) }{ \theta_1 (\Lambda W_{\bell_1}) } \right).\label{orL1}
\end{align}
By the decay properties of $W_{\bell}$ and integration by parts, we note that 
\begin{equation}
\label{oldf1}
\left|  \pshbbun{\eta}{ \theta_1 (\Lambda W_{\bell_1}) }\right|
+ \left|\pshbbun{\e}{\frac{\bell_1}{\lambda_1}\cdot \theta_1(\nabla  (\Lambda W_{\bell_1}))} \right| \lesssim \|\fl \e\|_E.
\end{equation}
Next, by \eqref{Tbbeta},
\begin{align*}
   \pshbbun{\theta_1 (\Lambda W_{\bell_1}) }{ \theta_1 (\Lambda W_{\bell_1}) } -\pshbbun {\e}{  \left(\theta_1 (\Lambda^2 W_{\bell_1})\right) }
  =  (1-|\bell_1|^2)^{\frac 12} \|\Lambda W\|_{\dot H^1}^2 +O(\|\fl \e\|_E),
\end{align*}
and by parity,
\begin{align*}
  \pshbbun{ \frac{\dot {\mathbf{y}}_1}{\lambda_1} \cdot\theta_1 (\nabla W_{\bell_1}) }{ \theta_1 (\Lambda W_{\bell_1}) }=0,\quad  \pshbbun {\e}{\frac{\dot {\mathbf{y}}_1}{\lambda_1} \cdot\theta_1 (\nabla\Lambda  W_{\bell_1})} =
O(|\dot {\mathbf{y}}_1|\|\fl \e\|_E ).
\end{align*}

Concerning the last   terms, we claim, for $k\in \{2,\ldots,K\}$,
\begin{equation}
\label{croises}
\left|  \frac{\dot\lambda_k}{\lambda_k}  \pshbbun{\theta_k (\Lambda W_{\bell_k}) }{ \theta_1 (\Lambda W_{\bell_1}) }\right| 
+ \left|  \pshbbun{ \frac{\dot {\mathbf{y}}_k}{\lambda_k} \cdot\theta_1 (\nabla W_{\bell_1}) }{ \theta_1 (\Lambda W_{\bell_1}) }   \right|
\lesssim \frac 1{t^3} \left(\left|\frac{\dot\lambda_k}{\lambda_k}\right|+\left|    \frac{\dot {\mathbf{y}}_k}{\lambda_k} \right|\right).
\end{equation}
Indeed, estimate \eqref{croises} is a direct consequence of the following technical result.
\begin{claim}\label{WW}
Let $0<r_2\leq r_1$ be such that $r_1+r_2> \frac 53$. For $t$ large, the following hold.
\begin{align}
& \hbox{ -- If $r_1>  \frac 53$ then}\quad   \int |W_1|^{r_1} |W_2|^{r_2}\lesssim t^{-3 r_2} ,\label{WW1}\\
& \hbox{ -- If $r_1\leq  \frac 53$ then}\quad  \int |W_1|^{r_1} |W_2|^{r_2} \lesssim t^{5-3 (r_1+r_2)}.\label{WW2}
\end{align}
\end{claim}
\begin{proof}[Proof of Claim \ref{WW}]
Estimates written in this proof are for $t$ large enough, and all constants may depend on $\bell_k$. For convenience, we denote
$$
\rho_k = x-\bell_k t - \mathbf{y}_{k}(t),\quad
\Omega_k(t)= \{ x \hbox{ such that } |\rho_k|< |\ell_1-\ell_2|t/10\}.
$$
Note that, for $t$ large,
\begin{align*}
\hbox{for $x\in \Omega_2$,} \quad |W_1(x)| \lesssim \frac 1{\langle \rho_1\rangle^3}
\lesssim \frac 1{(\langle \rho_2\rangle+t)^3} , \\
\hbox{for $x\in \Omega_2^C$,} \quad |W_2(x)| \lesssim \frac 1{(\langle \rho_2\rangle+t)^3}
\lesssim \frac 1{t^3},\\
\hbox{for $x\in \Omega_1^C$,} \quad |W_1(x)| \lesssim \frac 1{(\langle \rho_1\rangle+t)^3}
\lesssim \frac 1{t^3},
\end{align*}
\emph{Case $r_1> \frac 53$, $r_2>\frac 53$.} Then,
\begin{align*}
\int_{\Omega_2} |W_1|^{r_1}|W_2|^{r_2} \lesssim t^{-3 r_1} \int |W_2|^{r_2} \lesssim t^{-3r_1},
\quad\int_{\Omega_2^C} |W_1|^{r_1}|W_2|^{r_2}  \lesssim t^{-3 r_2} \int |W_1|^{r_1} \lesssim t^{-3r_2}.
\end{align*}
\emph{Case $r_1>  \frac 53$, $0<r_2\leq \frac 53$.} In this case, 
\begin{align*}
\int_{\Omega_2} |W_1|^{r_1}|W_2|^{r_2} 
& \lesssim  \int  \frac 1{(\langle \rho_2\rangle + t)^{3r_1}}  \frac 1{\langle \rho_2\rangle^{3r_2}} dx\\
& \lesssim \int 
\frac 1{(\langle x\rangle+ t)^{3r_1}} \frac {dx}{{\langle x\rangle}^{3r_2}} \lesssim t^{-3(r_1+r_2)+5} \lesssim t^{-3r_2},
\end{align*}
and
\begin{align*}
\int_{\Omega_2^C} |W_1|^{r_1}|W_2|^{r_2}  \lesssim t^{-3 r_2} \int |W_1|^{r_1} \lesssim t^{-3r_2}.
\end{align*}
\emph{Case $0<r_1\leq \frac 53$, $0<r_2\leq \frac 53$, $r_1+r_2> \frac 53$.}
First, as before,
\begin{align*}
\int_{\Omega_2} |W_1|^{r_1}|W_2|^{r_2} 
\lesssim t^{-3(r_1+r_2)+5},\quad 
\int_{\Omega_1} |W_1|^{r_1}|W_2|^{r_2}  \lesssim t^{-3(r_1+r_2)+5}.
\end{align*}
Next, by Holder inequality,
\begin{align*}
\int_{(\Omega_1\cup \Omega_2)^C} |W_1|^{r_1}|W_2|^{r_2} 
& \lesssim\left( \int_{\Omega_1^C} \frac 1{(\langle \rho_1\rangle +t )^{3(r_1+r_2)}}\right)^{\frac {r_1}{r_1+r_2}}
\left( \int_{\Omega_2^C}  \frac 1{(\langle \rho_2\rangle + t)^{3(r_1+r_2)}}\right)^{\frac {r_2}{r_1+r_2}}\\
& \lesssim t^{-3(r_1+r_2)+5}.
\end{align*}
The claim is proved
\end{proof}
In conclusion of the previous estimates,  the orthogonality condition
$\pshbbun { \e}{\theta_1 ({\Lambda W_{\bell_1})}}=0$, gives the following
\begin{equation}\label{el1}
|\dot \lambda_1 | \lesssim \noe {\fl \e} + |\dot {\mathbf{y}}_1 | \noe{\fl \e}
+ \frac 1{t^3}\sum_{k=1}^K \left(|\dot \lambda_k |+|\dot {\mathbf{y}}_k |\right).
\end{equation} 
Using the other orthogonality conditions, 
we obtain similarly, for $k=1,\ldots,5$,
\begin{align}\label{ey1}
&|\dot\lambda_k | \lesssim \noe {\fl \e} + |\dot {\mathbf{y}}_k |\noe{\fl \e}
+  \frac 1{t^3}\sum_{k'=1}^K \left(|\dot \lambda_{k'}|+|\dot {\mathbf{y}}_{k'}|\right),
\\
&|\dot {\mathbf{y}}_k | \lesssim \noe {\fl \e} + |\dot\lambda_k |\noe{\fl \e}
+  \frac 1{t^3}\sum_{k'=1}^K \left(|\dot \lambda_{k'}|+|\dot {\mathbf{y}}_{k'}|\right).\end{align} 
Combining these estimates, we find \eqref{le:p}. Note that equation \eqref{orL1} and the corresponding formula for $\dot \lambda_k$ and $\dot {\mathbf{y}}_k$ for $k\geq 1$, where $\fl \e$ is replaced by $\fl u- \sum_k \fl W_k$ form a nondegenerate first order differential system, whose unique solution is   $(\lambda_k,\mathbf{y}_k)_k$, which justifies the $C^1$ regularity of the parameters.
 
\medskip

{\bf Step 4.} Unstable directions. Recall that the quantities $z_k^\pm$ are defined through the $L^2$ scalar product $z_k^{\pm} (t)  = \psl {\fl \e(t)}{\fl {\tilde\theta}_k \fl Z_{\bell_k}^{\pm}}$. 
Recall also that $\fl Z_{\bell_k}^{\pm} \in \mathcal S$.
By \eqref{fle}, we have
\begin{align*}
&\frac d{dt} z_1^\pm   =\frac d{dt} \psl {\fl \e}{\fl {\tilde \theta}_1\fl Z_{\bell_1}^\pm} = 
\psl {\fl \e_t}{\fl {\tilde \theta}_1 \fl Z_{\bell_1}^{\pm}}+
\psl {\fl \e}{\partial_t\left(\fl {\tilde \theta}_1 \fl Z_{\bell_1}^{\pm}\right)} \\
 & = \psl {\fl {\mathcal L} \fl \e}{\fl {\tilde \theta}_1 \fl Z_{\bell_1}^{\pm}}
 + \frac {\bell_1}{\lambda_1}  \cdot  \psl{\fl \e}{\fl{\tilde \theta}_1 \nabla\fl Z_{\ell_1}^{\pm}}\\
&  +   \frac {\dot  \lambda_1}{\lambda_1}\left(  \psl{\fl \theta_1 \fl Z_{\bell_1}^{\Lambda}}{\fl {\tilde \theta}_1 \fl Z_{\bell_1}^{\pm}}
  -\psl{\fl \e}{\fl{\tilde \theta}_1 \fl\Lambda \fl Z_{\bell_1}^{\pm}} \right) 
  + \frac {\dot {\mathbf{y}}_1}{\lambda_1} \cdot \left( \psl{\fl \theta_1 \fl Z_{\bell_1}^{\nabla}}{\fl {\tilde \theta}_1 \fl Z_{\bell_1}^{\pm}}
  -\psl{\fl \e}{\fl{\tilde \theta}_1 \nabla\fl Z_{\bell_1}^{\pm}}\right)\\
& +  \sum_{k=2}^K \left( \frac {\dot \lambda_k}{\lambda_k}\psl {\fl\theta_k \fl Z_{\bell_k}^{\Lambda}}
{\fl {\tilde \theta}_1 \fl Z_{\bell_1}^{\pm}}
+ \frac {\dot {\mathbf{y}}_k}{\lambda_k} \cdot\psl{\fl \theta_k \fl Z_{\bell_k}^{\nabla}}{\fl {\tilde \theta}_1 \fl Z_{\bell_1}^{\pm}}\right)
+ \psl{\fl R_{\rm NL}+\fl R_{W}}{\fl {\tilde \theta}_1 \fl Z_{\bell_1}^{\pm}}.
\end{align*}
 
First, by direct computations, using \eqref{Zbbeta},
\begin{align*}
& \psl {\fl {\mathcal L} \fl \e}{\fl {\tilde \theta}_1 \fl Z_{\bell_1}^{\pm}}
 - \frac {\bell_1}{\lambda_1}   \cdot  \psl{\fl \e}{\fl{\tilde \theta}_1 \nabla\fl Z_{\bell_1}^{\pm}} \\
  & = \frac 1{\lambda_1} \psl{\fl \e}{\fl{\tilde \theta}_1 \left( -H_{\bell_1} J \fl Z_{\bell_1}^\pm\right)} + \sum_{k\geq 2} \psl \e {f'(\theta_k W_{\bell_k})  (\theta_1 Z_{\bell_1,2}^\pm)}\\
  & = \pm \frac {\sqrt{\lambda_0}}{\lambda_1} (1-|\bell_1|^2)^{\frac 12} z_1^\pm 
  +  \sum_{k\geq 2}\psl \e {f'(\theta_k W_{\bell_k})  (\theta_1 Z_{\bell_1,2}^\pm)}.
\end{align*}
Note that by the decay properties of $\fl Z_{\ell_1}^\pm$ and Claim \ref{WW}, for $k\geq 2$,
\begin{equation}
\label{bl}
	\left| \psl \e {f'(\theta_k W_{\bell_k})  (\theta_1 Z_{\bell_1,2}^\pm)}\right| \lesssim \frac {\noh \e}{t^4}.
\end{equation}

By \eqref{oZbb}, we have
$$
\psl{\fl \theta_1 \fl Z_{\bell_1}^{\Lambda}}{\fl {\tilde \theta}_1 \fl Z_{\bell_1}^{\pm}}= \psl{ \fl Z_{\bell_1}^{\Lambda}}{ \fl Z_{\bell_1}^{\pm}} =0,
$$
and thus, by \eqref{le:p},
\begin{equation}
\label{bl2}
\left|\frac {\dot\lambda_1}{\lambda_1}\left(  \psl{\fl \theta_1 \fl Z_{\ell_1}^{\Lambda}}{\fl {\tilde \theta}_1 \fl Z_{\ell_1}^{\pm}}
  -\psl{\fl \e}{\fl{\tilde \theta}_1 \fl \Lambda \fl Z_{\ell_1}^{\pm}} \right)\right| 
  \lesssim |\dot \lambda_1| \noe {\fl\e} \lesssim \noe {\fl\e}^2  .
\end{equation}
Similarly,
\begin{equation}
\label{bl3}
\left|\frac {\dot {\mathbf{y}}_1}{\lambda_1} \cdot \left( \psl{\fl \theta_1 \fl Z_{\ell_1}^{\nabla}}{\fl {\tilde \theta}_1 \fl Z_{\ell_1}^{\pm}}+\psl{\fl \e}{\fl{\tilde \theta}_1 \nabla\fl Z_{\ell_1}^{\pm}}\right)\right|\lesssim \noe {\fl\e}^2 .
\end{equation}

Next, by Claim \ref{WW}, we have
$$
\left|\psl {\fl\theta_2 \fl Z_{\ell_2}^{\Lambda}}
{\fl {\tilde \theta}_1 \fl Z_{\ell_1}^{\pm}}\right|
+\left|\psl{\fl \theta_2 \fl Z_{\ell_2}^{\nabla}}{\fl {\tilde \theta}_1 \fl Z_{\ell_1}^{\pm}}\right| \lesssim \frac 1{t^3}.
$$
Thus, by \eqref{le:p},
\begin{equation}
\label{bl4}
\left|\frac {\dot \lambda_2}{\lambda_2} \psl {\fl\theta_2 \fl Z_{\ell_2}^{\Lambda}}
{\fl {\tilde \theta}_1 \fl Z_{\ell_1}^{\pm}}\right|+
\left| \frac {\dot {\mathbf{y}}_2}{\lambda_2} \cdot\psl{\fl \theta_2 \fl Z_{\ell_2}^{\nabla}}{\fl {\tilde \theta}_1 \fl Z_{\ell_1}^{\pm}}\right| \lesssim 
\frac {\noe {\fl\e}}{t^3} .
\end{equation}

Finally, we claim
\begin{equation}
\label{Rnlb}
\left| \pse{\fl R_{W}}{\fl {\tilde \theta}_1 \fl Z_{\ell_1}^{\pm}}\right|+
\left| \pse{\fl R_{\rm NL}}{\fl {\tilde \theta}_1 \fl Z_{\ell_1}^{\pm}}\right|
\lesssim
\frac 1 {t^3} + \frac {\noh \e}t + \noh \e^2.
\end{equation}
\noindent Proof of \eqref{Rnlb}.
Note the following estimate, for any $p>1$,
\begin{equation}\label{general}
|R_W|=
\left|f\left(\SWk\right) - \sum_k f(W_k)\right| \lesssim \sum_{k\neq k'}|W_{k}|^{\frac 43}  |W_{k'}|.
\end{equation}
Thus, using Claim \ref{WW},
\begin{equation}
\label{Rnl1}
\left| \psl{R_W}{\theta_1 (-\ell_1 \partial_{x_1}\Lambda W_{\ell_1})}\right|\lesssim
\int \left( \sum_{k\neq k'} |W_k|^{\frac 43} |W_{k'}|\right) |W_1|^{\frac 43}
\lesssim \frac 1 {t^3} .
\end{equation}

Next, we decompose $R_{\rm NL}=  R_{\e,1}+ R_{\e,2},$ where 
\begin{align*}  
  R_{\e, 1} &=  \left( f'\left( \sum_k W_k\right)   - \sum_k f'\left( W_k\right)\right) \e,\\
 R_{\e, 2} & = f\left(\sum_k W_k   + \e\right) - f\left(\sum_k W_k\right) 
-   f'\left( \sum_k W_k\right)  \e.
\end{align*}
First, 
\begin{align*}
|R_{\e, 1}| & \leq \left(  \sum_{k'\neq k} |W_{k'}|^{\frac 13} |W_k|\right) |\e|.
\end{align*}
Thus, using Claim \ref{WW} and \eqref{z2}
\begin{align}
& \left|  \psl{R_{\e,1}}{\theta_1 (-\bell_1 \partial_{x_1}\Lambda W_{\ell_1})}\right|\lesssim
\int \left( \sum_{k\neq k'} |W_{k'}|^{\frac 13} |W_{k}|\right) |W_1|^{\frac 43} |\e| \\
& \lesssim  \left( \int |\e|^2 |W_1|^{\frac 23} \right)^{\frac 12} \left( \int
 W_1^2 \left(\sum_{k'\neq k} |W_{k'}|^{\frac 23} |W_k|^2 \right)\right)^{\frac 12}
 \lesssim \frac 1{t} \noh \e.
\end{align}
Finally, we have $|R_{\e,2}|\lesssim \left(\sum_k |W_k|^{\frac 13} \right)|\e|^2 
+ |\e|^{\frac 73}$, and thus, by \eqref{z2} and \eqref{z1},
\begin{align*}
& \left|  \psl{R_{\e,2}}{\theta_1 (-\bell_1 \partial_{x_1}\Lambda W_{\ell_1})}\right|\lesssim \int \left(\left(\sum_k |W_k|^{\frac 13} \right)|\e|^2 
+ |\e|^{\frac 73} \right)|W_1|^{\frac 43} \lesssim \noh \e^2 + \noh \e^{\frac 73}.
\end{align*}
The proof of  \eqref{Rnlb} is complete.

\medskip

Extending this computation to  $z_k^\pm$ for any $k$, we obtain in conclusion
\begin{equation}
\label{le:zb}
\left| \frac d{dt} z_k^{\pm}(t) \mp \frac{\sqrt{\lambda_0}}{\lambda_k(t)}(1-|\bell_k|^2)^{\frac 12}z_k^{\pm} (t)\right| \lesssim \noe {\fl\e(t)}^2 + \frac {\noe {\fl\e(t)}}t + \frac 1 {t^3}.
\end{equation}
The proof of Lemma \ref{le:4} is complete.
\end{proof}

\section{Proof of Theorem \ref{th:1} case (B)}\label{s4}
In this section,  we prove the existence of a solution $u(t)$ of \eqref{wave} satisfying \eqref{eq:th1}--\eqref{eq:th1bis} in case (B) of Theorem \ref{th:1}.
We argue by compactness and obtain $u(t)$ as the limit of suitable approximate multi-solitons $u_n(t)$. 

Let $K\geq 1$ and for  all $k\in\{1,\ldots,K\}$, let $\lambda^\infty_k>0$, ${\mathbf y}^\infty_k\in\R^5$ and $\bell_k \in \R^5$.
Let $S_n\to +\infty$.
For $\zeta_{k,n}^{\pm}\in \R$ small to be determined later (see statements of Proposition \ref{pr:s4}, Claim \ref{le:modu2} and Lemma \ref{le:bs2}), we consider the solution $u_n$ of
\begin{equation}\label{defun}
\left\{\begin{aligned} 
& \partial_t^2 u_n - \Delta u_n - |u_n|^{\frac 4{3}} u_n = 0 \\ 
& (u_n(S_n),\partial_t u_n(S_n))
 =   \sum_k  \left[  (\fl \theta_k^\infty \fl W_{\bell_k} ) (S_n)
 + \zeta_{k,n}^+ ( \fl \theta_k^\infty \fl Z_{\bell_k}^+)(S_n)
 + \zeta_{k,n}^-  (\fl \theta_k^\infty  \fl Z_{\bell_k}^-)(S_n)  \right]
 \end{aligned}\right.
\end{equation}
Note that since $(u_n(S_n),\partial_t u_n(S_n))\in Y^1\times Y^0$, the solution $u_n$ is well-defined in $Y^1\times Y^0$ at least on a small interval of time around $S_n$ (see section 2.1).

Now, we state the main uniform estimates on $u_n$.
\begin{proposition}\label{pr:s4}Under the assumptions of Theorem \ref{th:1}, case (B),
there exist $n_0>0$ and $T_0>0$  such that, for any $n\geq n_0$, there exist
$(\zeta_{k,n}^{\pm})_{k\in\{1,\ldots,K\}}\in \R^{2K}$, with
\begin{equation}\label{eq:choix}
 \sum_{k=1}^K |\zeta_{k,n}^{\pm}|^2 \lesssim \frac 1{S_n^{5}},
\end{equation}
and such that the solution $\fl u_n=(u_n,\partial_t u_n)$ of \eqref{defun} is well-defined in $Y^1\times Y^0$ on the time interval $[T_0,S_n]$
and satisfies
\begin{equation}\label{eq:un} 
\forall t\in [T_0,S_n],\quad  
 \left\|\fl u_n(t) 
-\sum_{k=1}^K  \fl W_k^\infty\right\|_{\dot H^1\times L^2}  \lesssim \frac 1{t},
\quad 
\left\|\fl u_n(t) 
-\sum_{k=1}^K  \fl W_k^\infty\right\|_{Y^1\times Y^0} 
\lesssim \frac 1{t^{\frac 12}}.
\end{equation}
\end{proposition}
\subsection{Proof of Theorem \ref{th:1} case (B), assuming Proposition \ref{pr:s4}}
In view of the uniform bounds obtained in \eqref{eq:un}  at $t=T_0$,
up to the extraction of a subsequence, $(u_n(T_0),\partial_t u_n(T_0))$ converges strongly
in $\dot H^1\times L^2$ to some $(u_0,u_1)$ as $n\to +\infty$.
Consider the solution $u(t)$ of \eqref{wave} associated to the initial data  $(u_0,u_1)$ at $t=T_0$.
Then, by the uniform bounds \eqref{eq:un} and the continuous dependence of the solution of \eqref{wave} with respect to its initial data in the energy space $\dot H^1\times L^2$ 
(see e.g. \cite{KM} and references therein), the solution $u$ is well-defined in the energy space  on $[T_0,\infty)$ and satisfies
\begin{equation}\label{eq:unbis}
\left\|\fl u(t) 
-\sum_{k=1}^K  W_k^\infty\right\|_{\dot H^1\times L^2} \lesssim \frac 1{t}.
\end{equation}
This finishes the proof of Theorem \ref{th:1} in case (B), assuming Proposition \ref{pr:s4}.

\medskip

The rest of this section is devoted to the proof of Proposition \ref{pr:s4}.

\subsection{Bootstrap setting}
We denote by $B_{\R^K}(\rho)$ (respectively, $S_{\R^K}(\rho)$)  the ball (respectively, the sphere) of $\R^K$ of center $0$ and of radius $\rho>0$, for 
the usual norm $|(\xi_k)_k|=\left(\sum_{k=1}^K \xi_k^2\right)^{1/2}$. 

For $t=S_n$ and for $t<S_n$ as long as $u(t)$ is well-defined in $\dot H^1\times L^2$ and satisfies \eqref{hyp:4}, we decompose $u_n(t)$ as in Lemma \ref{le:4}. In particular, we denote by $(\e,\eta)$, $(\lambda_k)_k$, $(\mathbf{y}_k)_k$, $(z_k^{\pm})_k$ the parameters of the decomposition of $u_n$.
We also set
\begin{equation}\label{defWK}
 \sumk = \sum_{k=1}^K W_k, \quad  \osumk = \sum_{k=1}^K |W_k|.
\end{equation}

We start with a technical result similar to Lemma 3 in \cite{CMM}. This claim will allow us to adjust the initial values of
$(z_k^\pm(S_n))_k$  from the choice of $\zeta_{k,n}^\pm$ in \eqref{defun}.

\begin{claim}[Choosing the initial  unstable modes] \label{le:modu2}
There exist $n_0>0$ and  $C>0$ such that, for all $n\geq n_0$, for any
$(\xi_k)_{k\in\{1,\ldots,K\}}\in \overline B_{\R^K}(S_n^{-5/2})$, there exists
a unique $(\zeta_{k,n}^{\pm})_{k\in\{1,\ldots,K\}}\in B_{\R^K}(C S_n^{-5/2})$ such that
the decomposition of $u_n(S_n)$ satisfies
\begin{equation}\label{modu:2}
z_k^-(S_n)=\xi_k,\quad
z_k^+(S_n)=0,
\end{equation}
\begin{equation}\label{modu3}
|\lambda_k(S_n)-\lambda_k^\infty|+
|\mathbf{y}_k(S_n)-\mathbf{y}_k^\infty|
+ \noe {\fl \e(S_n)}\lesssim S_n^{-5/2},
\end{equation}
\begin{equation}\label{modu5}
\|\fl \e(S_n)\|_{Y^1\times Y^0} \lesssim S_n^{-2}.
\end{equation}
\end{claim}
\begin{proof}[Sketch of the proof of Claim \ref{le:modu2}]
The proof of existence of $(\zeta_{k,n}^{\pm})_k$ in Claim \ref{le:modu2} is similar to  Lemma~3 in \cite{CMM} and we omit it.
Estimates in \eqref{modu3} are consequences of \eqref{bounds},
 \eqref{modu5} follows from \eqref{diff4}.
\end{proof}

From now on, for any $ (\xi_{k})_k\in \overline B_{\R^K}(S_n^{-5/2})$,
we fix  $(\zeta_{k,n}^{\pm})_k$ as given by Claim \ref{le:modu2} and  the corresponding solution $u_n$ of \eqref{defun}.

\medskip

The proof of Proposition \ref{pr:s4} is based on the following bootstrap estimates:
 for $C^*>1$ to be chosen,
\begin{equation}\label{eq:BS}
\left.\begin{aligned}
\sum_{k=1}^K|\lambda_k(t)-\lambda_k^\infty|+|\mathbf{y}_k(t)-\mathbf{y}_k^\infty|\leq \frac{(C^*)^2}{t}, \quad 
\sum_{k=1}^K |z_k^\pm(t)|^2\leq \frac 1{t^5}&\\ 
\noe{\fl \e(t)}\leq \frac{C^*}{t^2},\quad
\|\fl \e(t)\|_{Y^1\times Y^0}\leq \frac{(C^*)^2}{t^{\frac 12}}
& \\
\end{aligned}\right\}
\end{equation}
Set
\begin{equation}\label{def:tstar}
T^*=T_n^*((\xi_{k})_k)=\inf\{t\in[T_0,S_n]\ ; \ \hbox{$u_n$ satisfies \eqref{hyp:4} and (\ref{eq:BS}) holds on $[t,S_n]$} \} .
\end{equation}
Note that by Claim \ref{le:modu2},
estimate  \eqref{eq:BS} is satisfied  
at $t=S_n$.
Moreover,  if \eqref{eq:BS} is satisfied
on $[\tau,S_n]$ for some $\tau\leq S_n$ then by the well-posedness theory in $Y^1\times Y^0$
and continuity, 
$u_n(t)$ is well-defined and satisfies the decomposition of Lemma \ref{le:4} on $[\tau',S_n]$, 
 for some $\tau'<\tau$. 
In particular,  the definition of $T^*$ makes sense and it will suffice to strictly improve \eqref{eq:BS} on $[T^*,S_n]$ to prove $T^*=T_0$ for some ($\xi_k)_k$. Note also that we will prove  that $T^*=S_n$ for  $(\xi_k)_k\in S_{\R^K}(S_n^{-5/2})$
(see proof of Lemma \ref{le:bs2}).

\medskip

In what follows, we will prove that there exists $T_0$ large enough and at least one choice of $(\xi_{k})_{k}\in B_{\R^K}(S_n^{-5/2})$ so that $T^*=T_0$, which is enough to finish the proof of Proposition \ref{pr:s4}.
For this, we  derive general estimates for any $(\xi_k)_k\in \overline B_{\R^K}(S_n^{-5/2})$ (see Lemma \ref{le:bs1}) and use a   topological argument (see Lemma \ref{le:bs2}) to control the instable directions, in order to strictly improve estimates in \eqref{eq:BS} and thus prove that they cannot be saturated on $[T_0,S_n]$.

\subsection{Energy functional}
One of the main points of the proof of Proposition \ref{pr:s4} is to derive suitable estimates in the energy norm that will strictly improve the bound on  $\noe{\fl \e(t)}$ from \eqref{eq:BS}; the other estimates then follow easily.
 
 We claim the following proposition in  case  (B) of Theorem \ref{th:1}.
This is the only place in the paper where we need the restriction of collinear speeds.
 
\begin{proposition}\label{mainprop}
Under the assumptions of Theorem \ref{th:1}, case (B),
there exist $\mu>0$ and a function $\mathcal H_K(t)$ on $[T^*,S_n]$, which satisfies the following properties.\\
\emph{(i) Bound.}
\begin{equation}\label{boun}
|\mathcal H_K(t)| \leq \frac {\|\fl \e\|_E^2}{\mu}.
\end{equation}
\emph{(ii) Coercivity.}
\begin{equation}\label{coer}
\mathcal H_K(t) \geq \mu\|\fl \e\|_E^2 - \frac {t^{-5}}{\mu}  .
\end{equation}
\emph{(iii) Time variation.}
\begin{equation}\label{time}
- \frac d{dt} \left( t^{2} \mathcal H_K\right) (t) \lesssim C^* t^{-3}.
\end{equation}
\end{proposition}
\begin{proof}[Proof of Proposition \ref{mainprop}]
We consider the case where  the $K$ solitons are moving in the same direction. In particular, by rotation invariance, we   assume  
 \begin{equation}\label{hyp:sec3}
 \forall k\in\{1,\ldots,K\},\quad 
 \bell_k = \ell_k \mathbf{e}_1 \quad \hbox{where}\quad \ell_k \in (-1,1).
 \end{equation}
Moreover, without loss of generality,
 $$
 -1<\ell_1<\ldots<\ell_K<1.
$$
Fix
$$
\max_k (|\beta_k|)  < \overline \ell <1.
$$
For 
$$
0<\sigma<\frac 1{10} \min(\ell_{k+1}-\ell_k)
$$
small enough to be fixed, we set
\begin{align*}
 \hbox{for $k=1,\dots, K-1$},\quad 		&\ell_k^{+}=\ell_k+\sigma(\ell_{k+1}-\ell_k),\\
  	\hbox{for $k=2,\dots, K$},\quad  	&\ell_k^{-}=\ell_k-\sigma(\ell_{k}-\ell_{k-1}),
\end{align*}
and for $t>0$,
$$
\Omega(t) =  ( (\ell_1^+ t,\ell_{2}^- t)\cup\ldots \cup (\ell_{K-1}^+ t,\ell_{K}^- t)) \times \R^4,\quad
\Omega^C(t) = \R^5\setminus \Omega(t).
$$
We consider the continuous function $\chi_K(t,x)=\chi_K(t,x_1)$ defined as follows, for all $t>0$,
\begin{equation}\label{defchiK}
\left\{\begin{aligned}
    &\hbox{$\chi_K(t,x) =  \ell_1$ for $x_1\in (-\infty,  \ell_1^+ t]$},\\
	& \hbox{$\chi_K(t,x) = \ell_k $ for $x_1\in [\ell_k^- t, \ell_k^+ t]$, for $k\in \{2,\ldots,K-1\}$,}
  	\\ & \hbox{$\chi_K(t,x)= \ell_K$ for $x_1\in [\ell_{K}^- t,+\infty)$},\\
    	& \chi_K(t,x) = \frac{x_1}{(1-2\sigma)t}  - \frac {\sigma}{1-2 \sigma} (\ell_{k+1}+\ell_k) 
	\hbox{ for $x_1 \in [\ell_k^+ t,\ell_{k+1}^-t ]$, $k\in \{1,\ldots,K-1\}$}.
\end{aligned}
\right.
\end{equation}
In particular,
\begin{equation}\label{derchi}\left\{\begin{aligned}
& \partial_t \chi_K(t,x) =0,\quad  \nabla \chi_K(t,x)=0, \quad \hbox{on $\Omega^C(t)$},\\
   & \partial_{x_1} \chi_K(t,x)= \frac{1}{(1-2\sigma)t}  \quad \hbox{for $x\in \Omega(t)$},\\
   	& \partial_{t} \chi_K(t,x)= -\frac 1t \frac{x_1}{(1-2\sigma)t} \quad \hbox{for $x\in \Omega(t)$}.
\end{aligned} \right.\end{equation}
We define 
\begin{align*}
  & \mathcal H_K(t) =\int {\mathcal E}_K(t,x) dx 
  +    2 \int \left(\chi_K(t,x) \partial_{x_1} \e(t,x)\right) \eta(t,x) dx,
\end{align*}
where 
\begin{equation}\label{defeK}
 {\mathcal E}_K =|\nabla\e|^2+|\eta|^2- 2 \left(F\left(\sumk +\e\right)-F\left(\sumk\right)-f\left(\sumk\right)\e\right).
\end{equation}

\smallskip

Note   that from \eqref{eq:BS} and \eqref{le:p}, we have
$$
\sum_k \left( |\dot \lambda_k| + |\dot {\mathbf{y}}_k| \right) \lesssim \noe{\fl \e(t)}  \lesssim \frac {C^*}{t^2}.
$$
In particular, from \eqref{Mod_e} and \eqref{Mod_eta}, for all $p \in \N^5$ (here $|p|=\sum_j p_j$),
\begin{equation}\label{Idem}
  |\partial_x^p{\rm Mod}_\e(t)|\lesssim \frac{C^*}{t^2} \osumk^{1+\frac{|p|}{3}},\quad
  |\partial_x^p{\rm Mod}_\eta(t)|\lesssim \frac{C^*}{t^2} \osumk^{\frac 43+\frac{|p|}{3}}.
\end{equation}

\smallskip
 \noindent\emph{Proof of \eqref{boun}.}
 Since
 $$
| F(\sumk+\e) - F(\sumk) - f(\sumk) \e|\lesssim |\e|^{\frac {10}3} + \osumk^{\frac 43} |\e|^2,
 $$
the estimate \eqref{boun} on $\mathcal H_K$ follows from H\"older inequality, \eqref{z1} and \eqref{eq:BS}.

\smallskip

\noindent\emph{Proof of \eqref{coer}.}
Set
\begin{align*}  
   \Nint(t)   = \int_{\Omega} \left( |\nabla \e(t)|^2 + \eta^2(t) + 2   (\chi_K(t)\partial_{x_1} \e(t) ) \eta(t)\right)
\end{align*}
and
\begin{align*}  
   \Nsol(t)   = \int_{{\Omega^C}} \left( |\nabla \e(t)|^2 + \eta^2(t) \right).
\end{align*}
Note that, since $|\chi_K|<\overline \ell$,
\begin{equation}\label{nint} 
\begin{aligned}
\Nint  & =   \overline \ell \int_{\Omega} \left|\frac {\chi_K}{\overline \ell} \partial_{x_1} \e + \eta\right|^2  
+\int_{\Omega} |\overline \nabla \e|^2
+ \int_{\Omega} \left( 1-  \frac {\chi_K^2}{\overline \ell} \right) (\partial_{x_1} \e)^2 
+ (1-\overline \ell) \int \eta^2\\
  &  \geq \overline \ell \int_{\Omega} \left|\frac {\chi_K}{\overline \ell} \partial_{x_1} \e + \eta\right|^2  
    + (1-\overline \ell) \int_{\Omega}\left( |\nabla \e|^2 + \eta^2\right).
\end{aligned}
\end{equation}  
To obtain \eqref{coer}, we will actually prove the following stronger property
\begin{align}
  \mathcal H_K(t) &  \geq 
 \Nint(t) + \mu \Nsol(t) -\frac {t^{-5}}\mu  - \frac{t^{-4\alpha}}\mu \noe {\fl \e}^2 - \frac{1}\mu \noe {\fl \e}^3.\label{lF}
\end{align}
We decompose  
$
\mathcal H_K =   {\bf f_1}+ {\bf f_2} + {\bf f_3},
$
where
\begin{align*}
  {\bf f_1}  
& =   \int |\nabla \e|^2 -  \int \left(\sum_k f'(W_k)\right)\e^2  +  \int \eta^2 + 2\int  (\chi_K\partial_{x_1} \e) \eta,
\end{align*}
\begin{align*}
{\bf f_2} 
& =  - 2 \int \left( F\left(\sumk +\e\right)-F\left(\sumk\right) -f\left(\sumk\right)\e - \frac 12 f' \left(\sumk\right) \e^2 \right),
\end{align*}
\begin{align*}
{\bf f_3} 
& =    \int \left( \sum_k f'(W_k) -  f' \left(\sumk\right)  \right) \e^2 ,
\end{align*}
We claim the following estimates
\begin{align}
&  {\bf f_1} \geq  \Nint + \mu \Nsol - \frac {t^{-5}}{\mu} - \frac{t^{-4\alpha}}\mu \noe {\fl \e}^2,   \label{ff2}\\
& |{\bf f_2}|+|{\bf f_3}|\lesssim \noe{\fl\e}^3+\frac {\noe{\fl\e}^2}{t^2}.\label{ff3}
\end{align}
Note that combining these estimates with \eqref{eq:BS} and taking $T_0$ large enough (depending on $C^*$), we obtain \eqref{lF} and then \eqref{coer} for some other $\mu>0$.
\medskip

\noindent Proof of \eqref{ff2}. The main ingredient in the proof of \eqref{ff2} is Lemma \ref{pr:22}. For $\varphi$ defined in \eqref{phia}, set
$$
\varphi_k(t,x) = \varphi\left(\frac {x-\ell_k \mathbf{e}_1 t - \mathbf{y}_{k}(t)}{\lambda_k(t)}\right).
$$
We decompose $\bf f_1$ as follows
\begin{align*}
{\bf f_1}  & =   \Nint + \sum_k \left(\int |\nabla \e|^2\varphi_k^2 -  \int f'(W_k)\e^2 +  \int \eta^2 \varphi_k^2+ 2  \int  (\chi_K \partial_{x_1} \e) \eta\varphi_k^2 \right)\\
&+  \int_{\Omega^C} \left(|\nabla \e|^2+\eta^2+2 \chi_K(\partial_{x_1} \e)\eta\right)\left(1-\sum_k \varphi_k^2\right) \\
&- \int_{\Omega} \left(|\nabla \e|^2+\eta^2+2 \chi_K(\partial_{x_1} \e)\eta\right)\left(\sum_k \varphi_k^2\right)\\
& + 2 \sum_k \int (\chi_K- \ell_k) (\partial_{x_1} \e) \eta \varphi_k^2
=   \Nint+ {\bf f_{1,1}}   +{\bf f_{1,2}} +{\bf f_{1,3}}+{\bf f_{1,4}}.
\end{align*}

By Lemma \ref{pr:22}, the orthogonality conditions on $\fl \e$ and a change of variable, we have
\begin{align*}
{\bf f_{1,1}}  & \geq   \mu  \int \left( |\nabla \e|^2+\eta^2\right) \left(\sum_k\varphi_k^2\right)
-\frac 1{\mu} \sum_k \left( (z_k^-)^2 + (z_k^+)^2\right).
\end{align*}
Thus, using \eqref{eq:BS},
\begin{align*}
{\bf f_{1,1}}  & \geq   \mu  \int \left( |\nabla \e|^2+\eta^2\right) \left(\sum_k\varphi_k^2\right)
-\frac 1{\mu} \frac 1{t^5}
\geq \mu \int_{\Omega^C} \left( |\nabla \e|^2+\eta^2\right) \left(\sum_k\varphi_k^2\right)
-\frac 1{\mu} \frac 1{t^5}.
\end{align*}
Next, note that if $x$ is such that
$\varphi_k(t,x) >\frac 12$, then $\varphi_{k'}^2(x) \lesssim t^{-4 \alpha}$ for $k'\neq k$.
Thus, $$1-\sum_k \varphi_k^2  \gtrsim  - t^{-4 \alpha}.$$
By direct computations (with the notation $v_+ = \max (0,v)$),
\begin{align*}
{\bf f_{1,2}} &  =   \overline \ell \int_{\Omega^C} \left|\frac{\chi_K}{\overline \ell} \partial_{x_1} \e + \eta\right|^2 \left(1-\sum_k \varphi_k^2\right)  +\int_{\Omega^C} |\overline \nabla \e|^2\left(1-\sum_k \varphi_k^2\right) \nonumber\\
&+ \int_{\Omega^C} \left( 1- \frac{ \chi_K^2}{\overline\ell}\right) |\partial_{x_1} \e|^2\left(1-\sum_k \varphi_k^2\right) 
+ (1-\overline\ell) \int_{\Omega^C} \eta^2 \left(1-\sum_k \varphi_k^2\right)
\\ & \geq (1-\overline \ell) \int_{\Omega^C}  \left(| \nabla \e|^2 + \eta^2\right)\left(1-\sum_k \varphi_k^2\right)_+ - \frac  { \noe {\fl \e}^2}{t^{4\alpha}}.  
\end{align*}
Also, we see easily that 
$
|{\bf f_{1,3}}| \lesssim   t^{-4\alpha} \noe {\fl \e}^2. 
$

Finally, by the definition of $\chi_K$ in \eqref{defchiK}, the decay property of $\varphi$ and \eqref{eq:BS} (for a bound on $\mathbf{y}_k$), we have
$$
\|(\chi_K - \ell_k)\varphi_k\|_{L^\infty} \leq t^{-4 \alpha}.
$$
Thus,
$$
|{\bf f_{1,4}}| \lesssim t^{-4\alpha} \noe {\fl \e}^2 .
$$

Therefore, for some $\mu >0$, and $T_0$ large enough, we have
$$
{\bf f_{1,1}}+{\bf f_{1,2}}+{\bf f_{1,3}}+{\bf f_{1,4}}
\geq \mu \Nsol  - \frac 1\mu \frac 1{t^5} - t^{-4\alpha} \noe {\fl \e}^2 .
$$

\noindent Proof of \eqref{ff3}.
Using H\"older inequality, \eqref{z1} and \eqref{eq:BS}, we have
\begin{align*}
|{\bf f_{2}}|&\lesssim\int|\e|^{\frac{10}3}+|\e|^3\osumk^{\frac13} \lesssim \noe{\fl\e}^{3}.
\end{align*}
Next, since by the decay property of $W$,
$$
\left|f'(\sumk)- \sum_{k} f'(W_{k})  \right|\lesssim 
\frac{\osumk^{\frac 23}}{t^2},
$$
using \eqref{z2}, we obtain
\begin{align*}
|{\bf f_{3}}|&   \lesssim\frac 1{t^2} \int|\e|^2\osumk^{\frac23}      \lesssim\frac{\noe{\fl\e}^{2}}{t^2}.
\end{align*}

\emph{Proof of \eqref{time}.} {\bf Step 1.} First estimates.
We decompose
$$
\frac d{dt} \mathcal H_K = \int \partial_t {\mathcal E}_K + 2 \int \chi_K \partial_t\left( (\partial_{x_1} \e) \eta  \right) 
+ 2 \int (\partial_t \chi_K) (\partial_{x_1} \e) \eta = {\bf g_1} + {\bf g_2} + {\bf g_3}.
$$
We claim the following estimates
\begin{align}\label{pg1}
{\bf g_1} & =   2 \int  \e \left( -\Delta {\rm Mod}_\e - f'(\sumk) {\rm Mod}_\e   \right) +2 \int \eta{\rm Mod}_\eta \nonumber\\
& + 2 \int \left( \sum_k \ell_k \partial_{x_1} W_k\right) \left( f(\sumk + \e) - f(\sumk)-f'(\sumk) \e\right)
  + O\left(\frac{C^*}{t^5}\right),
\end{align}
\begin{align}
{\bf g_2} & = 
 -\frac 1{(1-2\sigma) t} \int_{\Omega}
        \left(\eta^2 + (\partial_{x_1} \e)^2
        - |\overline \nabla \e|^2   \right)\nonumber\\
&  - 2\int   \chi_K(\partial_{x_1} \sumk )\left( f(\sumk + \e) - f(\sumk)-f'(\sumk) \e\right)\nonumber\\
&+2 \int (\chi_K \partial_{x_1} {\rm Mod}_\e) \eta - 2 \int  \e   \chi_K \partial_{x_1} {\rm Mod}_\eta
+ O\left(\frac{C^*}{t^5}\right),\label{pg2}
\end{align}
\begin{align}\label{pg3}
{\bf g_3} & = -\frac 2{(1-2\sigma)t}  \int_{\Omega} \frac{x_1}{t} \partial_{x_1} \e \eta.
\end{align}

\emph{Estimate on $\bf g_1$.}
From direct computations and the definition of ${\rm Mod}_\e$ in \eqref{Mod_e}, we have
\begin{align*}
{\bf g_1} & = 2 \int \left(\nabla \e_t \cdot \nabla \e + \eta_t \eta - \e_t\left( f(\sumk + \e) - f(\sumk) \right)\right)\\
& + 2 \int \left( \sum_k \ell_k \partial_{x_1} W_k\right) \left( f(\sumk + \e) - f(\sumk)-f'(\sumk) \e\right)\\
& + 2 \int {\rm Mod}_\e\left( f(\sumk + \e) - f(\sumk)-f'(\sumk) \e\right)  = {\bf g_{1,1}} + {\bf g_{1,2}} + {\bf g_{1,3}}.
\end{align*}
Using \eqref{syst_e} and integration by parts,
\begin{align*}
{\bf g_{1,1}} & = 2 \int \eta R_W  + 2 \int \left( \nabla \e\cdot \nabla {\rm Mod}_\e
-  \left( f(\sumk + \e) - f(\sumk) \right) {\rm Mod}_\e + \eta {\rm Mod}_\eta \right)
\end{align*}
By Cauchy-Schwarz inequality, \eqref{general} and then \eqref{eq:BS},
$$
\left|\int \eta R_W\right| \lesssim \|\eta\|_{L^2} \|R_W\|_{L^2} \lesssim \frac{\|\eta\|_{L^2}}{t^3} \lesssim \frac{C^*}{t^5}.
$$
Thus,
\begin{align*}
{\bf g_{1,1}}+{\bf g_{1,3}} & =   2 \int  \e \left( -\Delta {\rm Mod}_\e - f'(\sumk) {\rm Mod}_\e   \right) +2 \int \eta{\rm Mod}_\eta
+ O\left(\frac{C^*}{t^5}\right),
\end{align*}
and \eqref{pg1} follows.

\emph{Estimate on $\bf g_2$.}
\begin{align*}
{\bf g_2} & = 2 \int (\chi_K \partial_{x_1} \e_t) \eta+ 2 \int (\chi_K \partial_{x_1} \e) \eta_t \\
& = 2 \int (\chi_K \partial_{x_1} \eta) \eta  + 2 \int (\chi_K \partial_{x_1} \e) \left( \Delta \e + \left( f(\sumk + \e) - f(\sumk) \right)+ R_W\right)
  \\& +2 \int (\chi_K \partial_{x_1} {\rm Mod}_\e) \eta +  2 \int (\chi_K \partial_{x_1} \e) {\rm Mod}_\eta.
\end{align*}
Note that by integration by parts and \eqref{derchi}
\begin{align*}
 2 \int (\chi_K \partial_{x_1} \eta) \eta  + 2 \int (\chi_K \partial_{x_1} \e)  \Delta \e 
&  = - \int \partial_{x_1} \chi_K \left(\eta^2 +(\partial_{x_1} \e)^2- |\overline \nabla \e|^2\right)\\
&  =- \frac 1{(1-2\sigma) t} \int_{\Omega}
        \left(\eta^2 + (\partial_{x_1} \e)^2
        - |\overline \nabla \e|^2   \right).
\end{align*}
Next, we observe
\begin{align*}
   \int (\chi_K \partial_{x_1} \e) \left( f(\sumk + \e) - f(\sumk)  \e\right) 
&=  \int \chi_K \partial_{x_1} \left( F(\sumk + \e) - F(\sumk)-f(\sumk) \e\right)\\
& -   \int   \chi_K (\partial_{x_1} \sumk) \left( f(\sumk + \e) - f(\sumk)-f'(\sumk) \e\right).
\end{align*}
Moreover, integrating by parts and using  \eqref{derchi},
\begin{align*}
& -\int \chi_K \partial_{x_1} \left( F(\sumk + \e) - F(\sumk)-f(\sumk) \e\right)\\
& = \frac 1{(1-2\sigma) t} \int_{\Omega} \left( F(\sumk + \e) - F(\sumk)-f(\sumk) \e\right).
\end{align*}
Thus, by \eqref{eq:BS} and the decay of $W$,
$$
\left| \int \chi_K \partial_{x_1}  \left( F(\sumk + \e) - F(\sumk)-f(\sumk) \e\right)\right|
\lesssim \frac 1{t} \int_\Omega \left(|\e|^{\frac {10}3} + \sumk^{\frac 43} |\e|^2 \right)
\lesssim \frac 1{t^5}.
$$
Last, integrating by parts, 
\begin{align*}
  2 \int (\chi_K \partial_{x_1} \e) {\rm Mod}_\eta 
&   =  -2 \int (\chi_K  \e) \partial_{x_1} {\rm Mod}_\eta - 2 \int (\partial_{x_1}\chi_K)\e {\rm Mod}_\eta\\
& =  -2 \int (\chi_K  \e) \partial_{x_1} {\rm Mod}_\eta  + O\left(\frac{1}{t^5}\right).
\end{align*}
Indeed, by \eqref{eq:BS}, \eqref{derchi}, \eqref{Idem} and \eqref{z2},
\begin{align*}
\left|\int (\partial_{x_1}\chi_K)\e {\rm Mod}_\eta\right|
& \lesssim \frac{C^*}{t^3} \int_{\Omega} |\e| \osumk^{\frac 43}
\lesssim \frac{C^*}{t^3}\left( \int  |\e|^2 \osumk^{\frac 23}\right)^{\frac 12}\left(\int_\Omega \osumk^2\right)^{\frac 12} \lesssim \frac{(C^*)^2}{t^{\frac {11}2}} \lesssim \frac{1}{t^5}.
\end{align*}

\emph{Estimate on $\bf g_3$.} \eqref{pg3} is a consequence of \eqref{derchi}.

\medskip

{\bf Step 2.} Using cancellations and conclusion.
In conclusion of estimates \eqref{pg1}--\eqref{pg3},
$$\frac d{dt} \mathcal H_K   =   {\bf h_1}+{\bf h_2}+{\bf h_3}+{\bf h_4}+ O\left(\frac{C^*}{t^5}\right),$$
where
$$
{\bf h_1} = -\frac 1{(1-2\sigma) t} \int_{\Omega}
        \left(\eta^2 + (\partial_{x_1} \e)^2 +2\frac{x_1}{t} (\partial_{x_1} \e) \eta
        - |\overline \nabla \e|^2   \right), $$      
$${\bf h_2} = 2 \int  \left(\sum_k \left(\ell_k  - \chi_K\right) \partial_{x_1} W_k\right)   \left( f(\sumk + \e) - f(\sumk)-f'(\sumk) \e\right),$$
$${\bf h_3} = 2 \int \eta\left({\rm Mod}_\eta +  \chi_K \partial_{x_1} {\rm Mod}_\e\right),$$
        $${\bf h_4}= 2 \int  \e \left( -\Delta {\rm Mod}_\e - \chi_K \partial_{x_1} {\rm Mod}_\eta- f'(\sumk) {\rm Mod}_\e   \right).$$
First, by \eqref{nint} and the definition of $\chi_K$ in \eqref{defchiK},
\begin{align*}
-((1-2\sigma)t) {\bf h_1} 
& \leq  \overline\ell \int_{\Omega} \left| \frac{\chi_K}{\overline \ell}  \partial_{x_1} \e + \eta\right|^2 + \int_{\Omega} \left( 1-  \frac{\chi_K^2}{\overline \ell}  \right) (\partial_{x_1} \e)^2 + (1- {\overline \ell}) \int \eta^2\\
  & + 2 \int_{\Omega} \left( \frac{x_1}t- \chi_K\right) \partial_{x_1} \e \eta 
   \leq \Nint + C\sigma \int \left(|\partial_{x_1} \e|^2 + \eta^2\right)
  \leq (1+C \sigma) \Nint.
\end{align*}
Second, we observe that by the definition of $\chi_K$ in \eqref{derchi} and the decay of $\partial_{x_1} W$ and $W$,
\begin{align*}
&  \left| \left(\ell_k  - \chi_K \right) \partial_{x_1}W_k  \right|\lesssim  \left|  \ell_k  - \chi_K \right| |W_k|^{4/3}   
\lesssim \frac 1{t^2} |W_k|^{2/3}.
 \end{align*}
Thus, by \eqref{z2} and \eqref{z1},
 \begin{align*}
 |{\bf h_2}|
 & 
 \lesssim \frac 1{t^2}\int \left( |\e|^{\frac {10}3}+|\e|^2\osumk^{\frac 23}\right)  \lesssim \frac {(C^*)^2}{t^6} \lesssim \frac 1{t^5}.
 \end{align*}
 
 Denote
 $$
 M_k = \frac{\dot \lambda_k}{\lambda_k} \Lambda W_k + \dot{\mathbf{y}}_k\cdot \nabla W_k \quad \hbox{so that}\quad
 {\rm Mod}_\e  = \sum_k M_k, \quad {\rm Mod}_\eta = - \sum_k \ell_k  \partial_{x_1} M_k
 $$
(see the   definition of ${\rm Mod}_\e$ and 
${\rm Mod}_\eta$ in \eqref{Mod_e}--\eqref{Mod_eta}).
Using \eqref{Idem}, the definition of $\chi_K$ (see \eqref{derchi}) and the decay of $W$,
\begin{equation}\label{MK}
|(\ell_k- \chi_K) \partial_{x_1} M_k|\lesssim \frac{C^*}{t^2} \frac 1{t^{\frac{13}{10}}} |W_k|^{\frac {9}{10}}.
\end{equation}
In particular,
$$
\left|{\rm Mod}_\eta +  \chi_K \partial_{x_1} {\rm Mod}_\e\right|\lesssim \frac {C^*}{t^{\frac {33}{10}}} \osumk^{\frac 9{10}},
$$
and thus, since $\osumk^{\frac 9{10}}$ is bounded in $L^2$, 
\begin{align*}
|{\bf h_3}|=\left| \int \eta\left({\rm Mod}_\eta +  \chi_K \partial_{x_1} {\rm Mod}_\e\right) \right|
\lesssim \frac {C^*}{t^{\frac {33}{10}}} \|\eta\|_{L^2}   \lesssim \frac {(C^*)^2}{t^{\frac {53}{10}}} \lesssim \frac 1{t^5}.
 \end{align*}
 
 Finally, we see that 
by \eqref{LW}, 
$
-\Delta M_k + \ell_k^2 \partial_{x_1}^2 M_k - f'(W_k) M_k=0. 
$ Thus, as before,
\begin{align*}
\left| -\Delta M_k + \ell_k \chi_K \partial_{x_1}^2 M_k -  f'(\sumk) M_k\right|
&\lesssim  \left| ( \chi_K -\ell_k) \partial_{x_1}^2 M_k\right| + \left|f'(\sumk)-f'(W_k)\right|  |M_k|\\
&\lesssim 
\frac{C^*}{t^{\frac{33}{10}}} |W_k|^{\frac{37}{30}}.
\end{align*}
Therefore,
$$
 \left| -\Delta {\rm Mod}_\e - \chi_K \partial_{x_1} {\rm Mod}_\eta- f'(\sumk) {\rm Mod}_\e   \right|
  \lesssim 
\frac{C^*}{t^{\frac{33}{10}}} \osumk^{\frac{37}{30}}.
$$
It follows that (by \eqref{z2}),
$$
|{\bf h_4}|\lesssim 
\frac{C^*}{t^{\frac{33}{10}}} \left\|\e  \osumk^{1/3}\right\|_{L^2} 
\lesssim
\frac{C^*}{t^{\frac{33}{10}}} \|\e\|_{\dot H^1}\lesssim  \frac{(C^*)^2}{t^{\frac{53}{10}}} \lesssim \frac 1{t^5}.
$$
 
 In conclusion, using \eqref{lF}, for $\sigma$ small, and $T_0$ large,
 $$
 -\frac d{dt} \mathcal H_K
 \leq  \frac {(1+C \sigma)} t \Nint + O\left(\frac{C^*}{t^5}\right)
 \leq \frac 2 t \mathcal H_K+ O\left(\frac{C^*}{t^5}\right).
$$

The proof   of Proposition \ref{mainprop}  is complete.
\end{proof}
\subsection{End of the proof of Proposition \ref{pr:s4}}\label{ssec:cc}
The following result, mainly based on Proposition \ref{mainprop}, improves all the estimates in \eqref{eq:BS}, except
the ones on $(z_k^-)_k$.
\begin{lemma}[Closing  estimates except $(z_k^-)_k$]\label{le:bs1}
For $C^*>0$ large enough,   for all $t\in[T^*,S_n]$,
\begin{equation}\label{eq:BSd}
\left.\begin{aligned}
|\lambda_k(t)-\lambda_k^\infty|+|\mathbf{y}_k(t)-\mathbf{y}_k^\infty|\leq \frac{(C^*)^2}{2 t},\quad
\sum_{k=1}^K |z_k^+(t)|^2\leq \frac 1{2 t^5}&\\ 
\noe{\fl \e(t)}\leq \frac{C^*}{2 t^2},\quad
\|\fl \e(t)\|_{Y^1\times Y^0}\leq \frac{(C^*)^2}{2 t^{\frac 12}}
& \\
\end{aligned}\right\}
\end{equation}
\end{lemma}
The control of the directions $(z_k^-)_k$, related to the dynamical instability  of $W$, requires a specific argument used in \cite{CMM} in a similar  context.
\begin{lemma}[Control of   unstable directions]\label{le:bs2}
There exist 
$(\xi_{k,n})_{k}\in B_{\R^K}(S_n^{-5/2})$ such that, for $C^*>0$ large enough,  $T^*((\xi_{k,n})_k)=T_0$.
In particular, let $(\zeta_n^\pm)$ be given by Claim \ref{le:modu2} from such $(\xi_{k,n})_{k}$,
then the solution $u_n$ of \eqref{defun} satisfies \eqref{eq:un}.
\end{lemma}

Note that   Lemma \ref{le:bs2} completes  the proof of Proposition \ref{pr:s4}.

\begin{proof}[Proof of Lemma \ref{le:bs1}]
{\bf Step 1.}  We prove that for $C^*$ large enough, for all $t\in [T^*,S_n]$, 
\begin{equation}\label{eq:BSt}
\|\fl \e\|_{Y^1\times Y^0}\leq \frac {(C^*)^2}{2 t^{\frac 12}}.
\end{equation}
The system  \eqref{syst_e} of equations of $\e$ and $\eta$ can be written under the form
$$
\left\{\begin{aligned}
&\e_t = \eta+{\rm Mod}_\e\\
&\eta_t = \Delta \e + R_\e +R_W + {\rm Mod}_{\eta},
\end{aligned}\right.
$$
where  
$$
|R_\e| \lesssim |\e|^{7/3} + |\e| \osumk^{4/3}, \quad
|\nabla R_\e| \lesssim |\nabla \e|\left( |\e|^{4/3} +  \osumk^{4/3}\right) + |\e| \osumk^{4/3}.
$$
In particular, by \eqref{esob}
$$
\|R_\e\|_{Y^0} \lesssim
\|\e\|_{\dot H^1}^{\frac 13} \|\e\|_{Y^1}^2 + t^{\frac 12} \|\e\|_{\dot H^1} \lesssim C^* t^{-3/2}.
$$
Moreover,
$$
\|R_W\|_{Y_0}\lesssim t^{-5/2},
$$
and by \eqref{Idem},
$$
\|{\rm Mod}_\e\|_{Y^1}+ \|{\rm Mod}_{\eta}\|_{Y^0}\lesssim   {C^*}{t^{-3/2}}.
$$

Using \eqref{modu5} and \eqref{EEn}, we obtain
\begin{align*}
\|\fl \e(t)\|_{Y^1\times Y^0}  &\lesssim
\|\fl \e(S_n)\|_{Y^1\times Y^0}\\
&+\int_t^{S_n} \left(\|R_\e(t')\|_{Y^0}+\|R_W(t')\|_{Y^0}+\|{\rm Mod}_\e(t')\|_{Y^1}+ \|{\rm Mod}_{\eta}(t')\|_{Y^0}\right) dt' \lesssim \frac{C^*}{t^{\frac12}} .
\end{align*}
In particular, taking $C^*$ large enough, we obtain \eqref{eq:BSt}.
 
\medskip

{\bf Step 2.} Estimates on parameters.
The estimates on $|\lambda_k(t)-\lambda_k^\infty|$ and $|\mathbf{y}_k(t)-\mathbf{y}_k^\infty|$ follow from integration of \eqref{le:p} using \eqref{eq:BS} and \eqref{modu3}, and possibly taking a larger $C^*$.
    
 Now, we prove the bound on $z_k^+(t)$. Let $c_k=\frac{\sqrt{\lambda_0}}{\lambda_k^{\infty}}(1-|\bell_k|^2)^{1/2}>0$.
  Then, from \eqref{le:z} and \eqref{eq:BS}, 
  $$
  \frac d{dt} \left[ e^{-c_k t} z_k^+\right]
  \lesssim e^{-c_k t} \frac{C^*}{t^3}. 
  $$
  Integrating on $[t,S_n]$ and using \eqref{modu:2}, we obtain
  $
  - z_k^+(t) \lesssim  C^*{t^{-3}}.
  $
  Doing the same for $-e^{-c_k t}  z_k^+$, we obtain the conclusion for $T_0$ large enough.
  
  \medskip
  
{\bf Step 3.} Bound on the energy norm.
Finally, to prove the estimate on $\noe{\fl \e(t)}$, we use Proposition \ref{mainprop}.
Recall from \eqref{modu3}  and then \eqref{boun} that
  \begin{equation}\label{init}
  \mathcal H_K(S_n) \lesssim  {S_n^{-5}}.
  \end{equation}
Integrating \eqref{time} on $[t,S_n]$, and using \eqref{init}, we obtain, for all $t\in [T^*,S_n]$,
$
	\mathcal H_K  \lesssim  {C^*}{t^{-4}} .
$
Using \eqref{coer}, we conclude that
$
\|\fl \e\|_E^2 \lesssim  {C^*}t^{-4}.$
\end{proof}

\begin{proof}[Proof of Lemma \ref{le:bs2}]
{\bf Step 1.} Choice of $(\zeta_{k})$. 
We follow the   strategy of Lemma 6 in \cite{CMM}. 
 The proof is by contradiction, we assume that for any $(\xi_k)_{k\in\{1,\ldots,K\}}\in B_{\R^K}(S_n^{-5/2})$,
$T^*((\xi_k)_k)$ defined by \eqref{def:tstar} satifies $T^*\in (T_0,S_n)$. In this case, by Lemma \ref{le:bs1} and continuity,
 it holds necessarily
\begin{equation}\label{eq:sat}
\sum_{k=1}^K |z_{k}^-(T^*)|^2= \frac{1}{(T^*)^{5}}.
\end{equation}

We claim the following transversality property at $T^*$
\begin{equation}\label{eq:trans}
 \left. \frac{d }{dt} \left(t^{5} \sum_{k=1}^K |z_{k}^-(t)|^2\right)\right|_{t=T^*} <-\overline c<0.
\end{equation}
Let $c_k=\frac{\sqrt{\lambda_0}}{\lambda_k^{\infty}}(1-|\bell_k|^2)^{1/2}>0$ and $\overline c = \min_k c_k$.
From \eqref{le:z} and \eqref{eq:BS}, for all $t\in [T^*,S_n]$,  
\begin{align*}
\frac{d}{dt}\left(t^5 \left(z_{k}^- \right)^2\right) & =
2 t^5  z_{k}^- \frac{d}{dt}z_{k}^-
+ 5 t^4  \left(z_{k}^- \right)^2   \\
& \leq 
- 2 t^5 c_k 
\left(z_{k}^- \right)^2 + \frac{C C^*}{t^{\frac 12}} \leq -2 \overline c t^5 \left(z_{k}^- \right)^2  + \frac{CC^*}{t^{\frac 12}}.
\end{align*}
Thus, from \eqref{eq:sat}
\begin{align*}
\left.\frac{d}{dt}\left(t^5 \sum_{k=1}^K \left(z_{k}^- \right)^2\right) \right|_{t=T^*}
& \leq 
-  2\overline  c  
+ \frac{C C^*}{(T^*)^{\frac 12}} < - \overline c , 
\end{align*}
  for $T_0$ large enough (depending on $C^*$, but independent of $n$).

\medskip

As a consequence of \eqref{eq:trans}, we observe that the map $T^*$
$$
(\xi_k)_{k\in\{1,\ldots,K\}}\in \overline B_{\R^K}(S_n^{-5/2}) \mapsto T^*((\xi_k)_k)
$$
is continuous. Indeed, if $T^*<S_n$, by \eqref{eq:trans}, it is clear that for all $\sigma>0$ small enough, there exists $\delta>0$ so that for all $t\in [T^*+\sigma,S_n]$, 
$t^5 \sum_k\left(z_{k}^-(t) \right)^2 <(1-\delta)$.
In particular, for $(\tilde\xi_k)_k\in B_{\R^K}(S_n^{-5/2})$ close enough to 
$(\xi_k)_k$, it follows that for all $t\in [T^*+\sigma,S_n]$, $t^5 \sum_k\left(\tilde  z_{k}^-(t) \right)^2 <(1-\frac 12\delta)$, and thus $\tilde T^*<T^*+\sigma$.  By similar arguments,  for 
$(\tilde \xi_k)_k\in B_{\R^K}(S_n^{-5/2})$ close enough to 
$(\xi_k)_k$,  we also have $\tilde T^*>T^*-\sigma$.

\medskip

We define  
\begin{align*}
\mathcal M \ : \ \overline  B_{\R^K}( S_n^{-5/2} ) &\to   S_{\R^K}( S_n^{-5/2} )\\
 (\xi_k)_k &\mapsto   \left(\frac {T^*}{S_n}\right)^{5/2} (z_k^-(T^*))_k 
\end{align*}
From what precedes, $\mathcal M$ is continuous. Moreover, from \eqref{eq:sat} and \eqref{eq:trans},
$\mathcal M$ restricted to $S_{\R^K}( S_n^{-5/2} )$ is the identity
(since in this case $T^*=S_n$ and $z_k^-(S_n)=\xi_k$ from \eqref{modu:2}). The existence of such a map is contradictory with Brouwer's fixed point theorem. 

\medskip

{\bf Step 2.} Conclusion. Proof of \eqref{eq:un}.
These estimates follow directly from the estimates \eqref{eq:BS} on $\e(t)$, $\lambda_k(t)$, $\mathbf{y}_k(t)$ and \eqref{diff}, \eqref{diff2}.
\end{proof}

\section{Proof of Theorem \ref{th:1} case (A) by   Lorentz transformation}

Let $\lambda_1^\infty$, $\lambda_2^\infty>0$, $\mathbf{y}_1^\infty$, $\mathbf{y}_2^\infty\in \R^5$, $\iota_1=\pm 1$,
$\iota_2=\pm 1$. Let $\bell_1, \bell_2 \in \R^5$ with $\bell_1\neq \bell_2$ and $|\bell_k|<1$ for $k=1,2$.
We claim that there exists a solution $u$ of \eqref{wave} in the energy space, on a time interval $[S_0,+\infty)$ such that
\eqref{eq:th1} and \eqref{eq:th1bis} hold.
 
\medskip

{\bf Step 1.} Reduction of the problem by rotation.
We change coordinates in $\R^5$ so that by invariance of \eqref{wave} by rotation, we reduce with loss of generality to the following case:
\begin{equation}\label{reduct}
 \bell_1 \cdot \mathbf{e}_1=\ell_1,\quad   \bell_2 \cdot \mathbf{e}_1= \ell_2,\quad 
\bell_1 \cdot \mathbf{e}_2 = \bell_2\cdot \mathbf{e}_2:=\beta,\quad  \bell_1 \cdot \mathbf{e}_j=  \bell_2 \cdot \mathbf{e}_j= 0, \hbox{ for $j=3,4,5$.}
\end{equation}
Indeed, it suffices to take as first vector of the new  orthonormal basis ${\mathcal B}'$ of $\R^5$, the vector
$\mathbf{e}_1' = \frac{\bell_1-\bell_2}{|\bell_1-\bell_2|}$,
and as second vector $\mathbf{e}_2' = a \bell_1 + b \bell_2$, where $a$ and $b$ are chosen so that $\mathbf{e}_1' \cdot \mathbf{e}_2'=0$ and $|\mathbf{e}_2'|=1$. Then, 
$\bell_1 \cdot \mathbf{e}_2'=\bell_2 \cdot \mathbf{e}_2'$. The basis $\mathcal B'$ is then completed in any way. 

Let $\overline{\overline x} = (x_3,x_4,x_5)$.

Note that if $\beta=0$, then $\bell_k = \ell_k \mathbf{e}_1$ for $k=1,2$ and then we are reduced to case (B) of Theorem~\ref{th:1} for $K=2$.
Now, we consider the general case $0<\beta<1$.
Set
\begin{equation}\label{vit}
	\tilde \ell_k = \frac {\ell_k}{\sqrt{1-\beta^2}},\quad |\tilde \ell_k|<1, \quad k=1,2.
\end{equation}
Also set ($k=1,2$)
\begin{equation}\label{param}
\tilde {\mathbf{y}}^{\infty}_k \in \R^5 \quad \hbox{such that}\quad 
\left\{\begin{aligned}
&
\tilde {\mathbf{y}}^{\infty}_{k,1} = \mathbf{y}^{\infty}_{k,1}+ \frac{\beta \ell_1}{ {1-\beta^2}} \mathbf{y}_{k,2}^\infty,\\
&
\tilde {\mathbf{y}}^{\infty}_{k,2} =   \frac{\mathbf{y}_{k,2}^\infty}{\sqrt{1-\beta^2}},\\
&
\tilde {\mathbf{y}}^{\infty}_{k,j} =   \mathbf{y}_{k,j}^\infty, \hbox{ for $j=3,4,5$}.
\end{aligned}\right.
\end{equation}
For $k=1,2$, let
$$
\tilde W_k^\infty (t,x) = 
\frac{\iota_k}{(\lambda_k^\infty)^{3/2}}W_{\tilde\ell_k}\left( \frac {x-\tilde \ell_k \mathbf{e}_1 t- \tilde {\mathbf{y}}_k^\infty}{\lambda_k^\infty}\right),\quad
\fl {\tilde W}_k^\infty  = (\tilde W_k^\infty,\partial_t \tilde W_k^\infty).
$$
Let $\tilde u(t)$ be the solution of \eqref{wave} satisfying
\begin{equation}\label{LIM}
\left\|\fl {\tilde u}(t) -\left[ \fl{\tilde W}_1^\infty(t)+\fl {\tilde W}_2^\infty(t)\right]\right\|_{\dot H^1\times L^2} = 0
\end{equation}
given by  Theorem \ref{th:1}, case (B).
Define the Lorentz transform with parameter $\beta \mathbf{e}_2$ of the solution $\tilde u$, i.e.
\begin{equation}\label{Lor}
u(s,y) = \tilde u\left(\frac{s-\beta y_2}{\sqrt{1-\beta^2}},y_1,\frac{y_2-\beta s}{\sqrt{1-\beta^2}}, \overline{\overline y}\right).
\end{equation}
We claim that $u(s,y)$ is a $2$-soliton of \eqref{wave}  in the sense of Theorem \ref{th:1} with parameters $\lambda_k^\infty,$ $\mathbf{y}_k^\infty$ and speeds
$\ell_k \mathbf{e}_1+ \beta \mathbf{e}_2$.

First, from the arguments of the proof of Lemma 6.1 in \cite{DKM2}, since $\tilde u(t,x)$ is well-defined on $[T_0,+\infty)$
it is well-defined everywhere on the space-time domain $\R\times \R^5$ except possibly in a half cone of the form 
$t-t^- < - |x-x^-|$, for some $t^-\in \R$ and $x^- \in \R^5$. Thus, there exists $S_0\in \R$ such that $u(s)$ defined by \eqref{Lor}
 makes sense on $\R^5$ for all
$s>S_0$ (see also Lemma \ref{DKM} below). Moreover, from the arguments of section 6 in \cite{DKM2} (see also section 2 of \cite{KM}),
$u$ is a finite energy solution of \eqref{wave} on $[S_0,+\infty)$.

To prove the claim, we consider separately the regions ``far from the solitons'' and ``close to the solitons''.

\medskip

{\bf Step 2.} Estimate  far from the solitons. We claim that for all $\delta>0$, there exists $A_\delta>0$ such that for all $s\geq S_\delta$,
\begin{equation}\label{far}
\|(u(s),\partial_t u(s))\|_{(\dot H^1\times L^2)(|y-(\ell_k \mathbf{e}_1 + \beta \mathbf{e}_2)|>A_\delta)}\lesssim \delta.
\end{equation}

Let $\delta>0$ and $T_\delta>0$ be such that 
\begin{equation}\label{petitun}
\sup_{t>T_\delta} \left\|\fl {\tilde u}(t) -\left[ \fl{\tilde W}_1^\infty(t)+\fl {\tilde W}_2^\infty(t)\right]\right\|_{\dot H^1\times L^2} 
<\delta.
\end{equation}
Moreover, let $A_\delta>1$ large enough so that for $k=1,2$,
\begin{equation}\label{petitdeux}
\sup_{t\in \R} \left\|\fl{\tilde W}_k^{\infty}(t)\right\|_{(\dot H^1\times L^2)(|x-\tilde \ell_k \mathbf{e}_1 t|>A_\delta/2)}<\delta.
\end{equation}

We recall the following result from section 2 of \cite{KM}, Claim 6.7 and   proof of Lemma~6.1 of \cite{DKM2} 
(and references therein for the small data Cauchy theory).

\begin{lemma}[Small scattering solutions and Lorentz transform \cite{DKM2}]\label{DKM}
There exists $\delta_0>0$ such that the following holds.
\begin{itemize}
\item[\rm (i)] For all $(w_0,w_1)\in \dot H^1\times L^2$ such that $\|(w_0,w_1)\|_{\dot H^1\times L^2}<\delta_0$,
there exists a global scattering solution\footnote{by global scattering solution, we mean a solution defined for all time $t\in \R$ and behaving in the energy space as a free solution both as $t\to +\infty$ and $t\to -\infty$} $(w(t),\partial_t w(t))$ of \eqref{wave} with initial data $(w_0,w_1)$.

Moreover, $\sup_{t\in \R} \|(w(t),\partial_t w(t))\|_{\dot H^1\times L^2}\lesssim \delta_0$.
\item[\rm (ii)] For $(w,\partial_t w)$ as in {\rm (i)} and $\beta\in (-1,1)$, the function $w_\beta(s,y)$ defined by
\begin{equation}\label{Lorw}
w_\beta(s,y) = w\left(\frac{s-\beta y_2}{\sqrt{1-\beta^2}},y_1,\frac{y_2-\beta s}{\sqrt{1-\beta^2}}, \overline{\overline y}\right)
\end{equation}
is a global scattering solution of \eqref{wave}. Moreover, for some constant $C_\beta >0$,
\begin{equation}\label{unif}
\sup_{t\in \R} \|(w_\beta,\partial_t w_\beta)(t)\|_{\dot H^1\times L^2}\leq C_\beta \|(w_0,w_1)\|_{\dot H^1\times L^2}.
\end{equation}
\end{itemize}
\end{lemma}

We defined a cutoff function $\zeta\in C^\infty(\R^5)$ such that
$$
\zeta(x) = 1 \hbox{ for $|x|>1$}, \quad \zeta(x)=0 \hbox{ for $|x|<\frac 12$}.
$$
For $t_0>T_\delta$ to be chosen later, we also define
$$
\zeta^{\rm ext}(x) = \zeta\left(\frac{x-\tilde \ell_1 \mathbf{e}_1 t_0}{A_\delta} \right) \zeta\left(\frac{x-\tilde \ell_2 \mathbf{e}_1 t_0}{A_\delta} \right).
$$
Define $u^{\rm ext}(t)$ the solution of \eqref{wave} corresponding to the following initial data at $t=t_0$,
$$
u^{\rm ext}(t_0,x) = \tilde u(t_0,x) \zeta^{\rm ext}(x),\quad
\partial_t u ^{\rm ext}(t_0,x) = (\partial_t \tilde u(t_0,x)) \zeta^{\rm ext}(x).
$$
By \eqref{petitun} and \eqref{petitdeux}, choosing $\delta>0$ small enough (compared to $\delta_0$, given by Lemma \ref{DKM}), we have
$$
\|(u^{\rm ext}(t_0),\partial_t u^{\rm ext} (t_0))\|_{\dot H^1\times L^2} \leq \delta <\delta_0.
$$
By Lemma \ref{DKM}, $u^{\rm ext}(t)$ is thus a global scattering solution of \eqref{wave} on $\R\times \R^5$, and satisfies
$$
\sup_{t\in \R} \|(u^{\rm ext}(t),\partial_t u^{\rm ext} (t))\|_{\dot H^1\times L^2} \lesssim \delta .
$$
Moreover, if we define $u_\beta^{\rm ext} (s,y)$ as the Lorentz transform  with parameter $\beta \mathbf{e}_2$ of $u^{\rm ext}$ (as in \eqref{Lorw}), then
$u_\beta^{\rm ext}$ is also a global scattering solution of \eqref{wave} satisfying
\begin{equation}\label{ext}
\sup_{s\in \R} \|(u_\beta^{\rm ext}(s),\partial_t u_\beta^{\rm ext} (s))\|_{\dot H^1\times L^2} \lesssim \delta.
\end{equation}

Now, we deduce consequences of these observations on $\tilde u$ and $u$.  Indeed, since
$
u^{\rm ext}(t_0,x)=u(t_0,x),$ and $ \partial_t u^{\rm ext}(t_0,x)=\partial_t u(t_0,x),$ for a.e. $(t,x)$ such that
$
|x-\tilde \ell_k \mathbf{e}_1 t_0|>A_\delta$ for $k=1,2$,
it follows from finite speed of propagation that 
$$
u^{\rm ext} (t,x) = \tilde u(t,x), \quad \partial_t u^{\rm ext} (t,x) = \partial_t  \tilde u(t,x) \quad \hbox{a.e. on } C_{A_\delta} (t_0),
$$
where
$$
C_{A_\delta} (t_0) = \{ (t,x) \hbox{ such that }
|x-\tilde \ell_1 \mathbf{e}_1 t_0|>A_\delta+|t-t_0| \hbox{ and }
|x-\tilde \ell_2 \mathbf{e}_1 t_0|>A_\delta+|t-t_0|\}.
$$
Then, by the definitions of $u$ and $u^{\rm ext}_\beta$, 
for almost every $(s,y)$ such that
$$ 
\left(\frac{s-\beta y_2}{\sqrt{1-\beta^2}},y_1,\frac{y_2 - \beta s}{\sqrt{1-\beta^2}}, \overline{\overline y}\right)\in C_{A_\delta}(t_0),
$$
we have
\begin{equation}\label{equal}
u^{\rm ext} (s,y) = u(s,y), \quad \partial_s u^{\rm ext}  (s,y) = \partial_t  u(s,y).
\end{equation}
Now, let $s_0\geq S_\delta := \frac{T_\delta}{\sqrt{1-\beta^2}}$ and choose $t_0 = \sqrt{1-\beta^2} s_0$.
By \eqref{ext} and \eqref{equal}, 
\begin{align}
\|(u_\beta(s_0), \partial_t u_\beta(s_0))\|_{(\dot H^1\times L^2)(\Omega_{A_\delta}(s_0))}
& = \|(u^{\rm ext}_\beta(s_0), \partial_t u^{\rm ext}_\beta(s_0))\|_{(\dot H^1\times L^2)(\Omega_{A_\delta}(s_0))}
\nonumber \\& \leq \|(u_\beta^{\rm ext}(s_0), \partial_t u_\beta^{\rm ext}(s_0))\|_{\dot H^1\times L^2}\lesssim \delta,\label{pourfinir}
\end{align} 
where
$$
\Omega_{A_\delta}(s_0) =
\left\{ y \hbox{ such that }
\left(\frac{s_0-\beta y_2}{\sqrt{1-\beta^2}},y_1,\frac{y_2-\beta s_0}{\sqrt{1-\beta^2}},\overline{\overline y}\right)\in C_{A_{\delta}}(t_0)\right\}.
$$
For $C_\beta = \frac 2{1-|\beta|}$, let
$$
\Gamma_{A_\delta}(s_0)=
\left\{ y \hbox{ such that }
\left(|y_1 - \ell_k s_0|^2+ |y_2- \beta s_0|^2+|\overline{\overline y}|^2\right) > C_\beta A_{\delta}\hbox { for $k=1$ and $2$}.
\right\}
$$
We claim that 
\begin{equation}\label{include}
\Omega_{A_\delta}(s_0) \supset \Gamma_{A_\delta}(s_0).
\end{equation}
Indeed, for $y\in \Gamma_{A_\delta}(s_0)$,  by the choice of $t_0$, for $k=1,2$,
\begin{align*}
&\left(|y_1 - \tilde \ell_k  t_0|^2 + \frac 1{1-\beta^2} |y_2- \beta s_0|^2 + |\overline{\overline y}|^2\right)^{1/2}
=\left(|y_1 - \ell_k s_0|^2 + \frac 1{1-\beta^2} |y_2- \beta s_0|^2 + |\overline{\overline y}|^2\right)^{1/2}\\
& \geq  (1-|\beta|) \left(|y_1 -   \ell_k s_0|^2 +   |y_2- \beta s_0|^2 + |\overline{\overline y}|^2\right)^{1/2}
+\frac {|\beta|}{\sqrt{1-\beta^2}} |y_2- \beta s_0| \\
& > A_\delta +  \frac {|\beta|}{\sqrt{1-\beta^2}} |y_2- \beta s_0|
= A_\delta + \left| \frac{s_0 - \beta y_2}{\sqrt{1-\beta^2}} - t_0\right|.\end{align*}
Thus, $y\in \Omega_{A_\delta}(s_0)$. 

Now, we observe that \eqref{include} and \eqref{pourfinir} prove \eqref{far}.

\medskip

{\bf Step 3.} Estimate close to the solitons.
First, we compute $W_k^\infty(s,y)$, the Lorentz transform with parameter $\beta \mathbf{e}_2$ of $\tilde W_k(t,x)$.
From the definition of $\tilde W_k^\infty$, \eqref{vit} and   \eqref{param},
\begin{align*}
&W_k^\infty(s,y) 
 = \tilde W_k^\infty\left(\frac{s-\beta y_2}{\sqrt{1-\beta^2}}, y_1, \frac{y_2-\beta s}{\sqrt{1-\beta^2}},\overline{\overline y}\right)\\
& = \frac  {\iota_k}{(\lambda_k^\infty)^{3/2}}W\left(\frac{y_1 - \tilde \ell_k\left(\frac{s-\beta y_2}{\sqrt{1-\beta^2}}\right) -\tilde {\mathbf{y}}_{k,1}^\infty}
{\lambda_k^\infty\sqrt{1-\tilde \ell_k^2}   }, \frac{\frac{y_2-\beta s}{\sqrt{1-\beta^2}} - \tilde {\mathbf{y}}_{k,2}^\infty}{\lambda_k^\infty},
\frac {\overline{\overline y}- \overline{\overline{\tilde {\mathbf{y}}}}_k^\infty}{\lambda_k^{\infty}}\right)\\
&=  \frac {\iota_k}{(\lambda_k^\infty)^{3/2}}
W\left(\frac{\left(y_1 - \ell_k s-  \mathbf{y}_{k,1}^\infty\right) 
+\frac{\beta\ell_k}{1-\beta^2}({y_2-\beta s} -   \mathbf{y}_{k,2}^\infty)}{\lambda_k^{\infty}\sqrt{1-\frac{\ell_1^2}{1-\beta^2}}} ,
 \frac{ {y_2-\beta s} -   \mathbf{y}_{k,2}^\infty} {\sqrt{1-\beta^2} \lambda_k^\infty},
\frac {\overline{\overline y}- \overline{\overline{  {\mathbf{y}}}}_k^\infty}{\lambda_k^{\infty}}\right).
\end{align*}
By the radial symmetry of $W$, i.e. $W(x)=W(|x|)$, we have 
$$
W_k^\infty(s,y) = 
\frac{\iota_k}{(\lambda_k^\infty)^{3/2}} W_{\ell_k \mathbf{e}_1+\beta \mathbf{e}_2} \left( \frac{y-(\ell_k \mathbf{e}_1 + \beta \mathbf{e}_2)s-\mathbf{y}_k^{\infty}}{\lambda_k^\infty}\right).
$$

Therefore, the Lorentz transform with parameter $\beta \mathbf{e}_2$ of 
$\tilde v=\tilde u-[\tilde W_1^\infty+\tilde W_2^\infty]$ is $v = u - [W_1^\infty+W_2^\infty]$ and to finish the proof of Theorem \ref{th:1} in case (A), we only have to prove that, for $S_\delta$ large enough,
\begin{equation}\label{limv}
\sup_{s>S_\delta}\|(v,\partial_s v)(s)\|_{\dot H^1\times L^2} \lesssim \delta.
\end{equation}
By \eqref{far} and the decay properties of $W$, we know that for $S_\delta$ large,
\begin{equation}\label{Newfar}
\sup_{s>S_\delta}\|(v,\partial_s v)(s)\|_{(\dot H^1\times L^2)(\Gamma_{A_\delta}(s))} \lesssim \delta.
\end{equation}
We now concentrate on an estimate for $v(s)$ close to the soliton centers. 

First, we  claim that for any $\delta>0$, for any $B>1$, for $S_\delta(\delta,B)$ large enough, and any $s_0>S_\delta$,
\begin{equation}\label{onv}
\iint_{|s-s_0|+ |y-(\ell_k \mathbf{e}_1 + \beta \mathbf{e}_2)s|<B} 
\left(|\nabla v|^2 + |\partial_s v|^2 \right)dyds \lesssim \delta.
\end{equation}
Indeed, by change of variables,
\begin{align*}
& \iint_{|s-s_0|+ |y-(\ell_k \mathbf{e}_1 + \beta \mathbf{e}_2)s|<B} |\partial_s v|^2 dyds\\
& =\frac1{ 1-\beta^2}
\iint_{|s-s_0|+ |y-(\ell_k \mathbf{e}_1 + \beta \mathbf{e}_2)s|<B}
\left| \left( \tilde v_t - \beta  \tilde v_{x_2}\right)
	\left(\frac{s-\beta y_2}{\sqrt{1-\beta^2}},y_1,\frac{y_2-\beta s}{\sqrt{1-\beta^2}},\overline{\overline y}\right)\right|^2 dyds.
\end{align*}
Changing variables in the integral on the right-hand side as follows (note that the Jacobian of the change of variable is $1$)
$$
t=\frac{s-\beta y_2}{\sqrt{1-\beta^2}}, \quad x_1 = y_1,\quad x_2 = \frac{y_2-\beta s}{\sqrt{1-\beta^2}}, 
\quad \overline{\overline x}=\overline{\overline y},
$$
we obtain, for some $C=C(\delta)$,
\begin{align*}
& \iint_{|s-s_0|+ |y-(\ell_k \mathbf{e}_1 + \beta \mathbf{e}_2)s|<B} 
 |\partial_s v|^2  dyds\\
& \qquad  \lesssim
\iint_{|t-s_0 \sqrt{1-\beta^2}|+ |x-\tilde \ell_k \mathbf{e}_1 t|<CB} \left(|\tilde v_t|^2 +|\tilde v_{x_2}|^2\right) dxdt
\\
& \qquad \lesssim B \sup_{t>s_0\sqrt{1-\beta^2} - CB} \int \left(|\tilde v_t|^2 +|\tilde v_{x_2}|^2\right)(t) dx \lesssim \delta,
\end{align*}
for $S_\delta(\delta,B)$ large enough by \eqref{LIM}.
Proceeding similarly for $|\nabla v|^2$, we obtain \eqref{onv}.

It follows from \eqref{onv} and \eqref{Newfar} that for any $s_0>S_\delta$, there exists $s_1 \in [s_0,s_0+1]$, such that
\begin{equation}\label{inits1}
\|(v, \partial_s v)(s_1)\|_{\dot H^1\times L^2}^2 \lesssim \delta.
\end{equation}
Now,  we use  the equation of $v$ to obtain an energy estimate  for all large time.
Note that $v$ satisfies
\begin{equation}\label{eqvb}
v_{tt}- \Delta v + f\left(v + W_1^\infty+W_2^\infty\right) - f\left( W_1^\infty\right) - f\left(W_2^\infty\right) = 0.
\end{equation} 
Using the equation of $v$, the properties of  $W_k^\infty$ and standard small data Cauchy theory (by Strichartz estimates, see e.g. section 2 of \cite{KM}), taking $\delta>0$ small enough, and for $S_\delta$ large enough,
we obtain from \eqref{inits1},
$$
\sup_{[s_1-1,s_1+1]}\int \left(|\nabla v|^2 + |\partial_s v|^2 \right)(s,y)dyds \lesssim \delta.
$$
Thus, \eqref{limv} is proved.

This completes the proof of Theorem \ref{th:1} in case (A).

\end{document}